\documentclass[12pt]{amsproc}

\usepackage{amsmath,amssymb,amsthm}
\usepackage[margin=1.15in]{geometry}
\usepackage{hyperref}
\usepackage{booktabs}
\usepackage{colonequals}
\usepackage{tikz}
\usetikzlibrary{shapes.geometric}

\tikzset{
  vtx/.style={circle,draw,fill=black,inner sep=1.6pt},
  lyr/.style={ellipse,draw,minimum width=16mm,minimum height=26mm},
  wt/.style={font=\small,inner sep=1.5pt,fill=white},
  del/.style={dashed,thick},
}

\newtheorem{theorem}{Theorem}[section]
\newtheorem{proposition}[theorem]{Proposition}
\newtheorem{lemma}[theorem]{Lemma}
\newtheorem{corollary}[theorem]{Corollary}
\theoremstyle{definition}
\newtheorem{remark}[theorem]{Remark}

\newtheorem{example}[theorem]{Example}

\newcommand{\dD}{d^{D}}
\newcommand{\Norm}[1]{\left\lVert #1 \right\rVert}
\DeclareMathOperator{\Pf}{Pf}

\DeclareMathOperator{\diag}{diag}

\renewcommand\le{\leqslant}
\renewcommand\ge{\geqslant}

\newcommand{\calK}{\mathcal{K}}

\newcommand{\calM}{\mathcal{M}}

\newcommand{\calP}{\mathcal{P}}

\newcommand{\PP}{\mathbb{P}}
\newcommand{\R}{\mathbb{R}}
\newcommand{\C}{\mathbb{C}}

\newcommand{\ep}{\varepsilon}

\newcommand{\PM}{\calP \calM}

\newcommand{\sgn}{{\rm sgn}}

\newcommand{\ii}{{\rm i}}

\newcommand{\Sym}{{\rm Sym}}
\newcommand{\Aut}{{\rm Aut}}

\title{The effect of edge deletion on noncommutative distances on graphs}

\author{Eric Swartz}
\address[E. Swartz]{Department of Mathematics, William \& Mary, Williamsburg, VA 23187}
\email{\textcolor{blue}{\href{mailto:easwartz@wm.edu}{easwartz@wm.edu}}}
\date{\today}

\begin{document}

\begin{abstract}
For a Dirac operator $D$ on a finite weighted graph, let $d^D$ denote the associated
noncommutative (Connes) distance. We show that deleting an edge can actually decrease the noncommutative distance between two vertices in the graph.  The smallest example of this phenomenon comes from a weighted 4-cycle, and the case of deleting an edge from a weighted 4-cycle is determined completely: if $d^{D'}$ denotes the noncommutative distance in the graph after edge deletion, we prove that $\sup d^D/d^{D'} = 2/\sqrt{3}$ for the $4$-cycle, and further that this upper bound holds for all graphs.  On the other hand, we show that deleting an edge can decrease the noncommutative distance between two vertices of an $n$-cycle precisely when $n$ is divisible by $4$.

In cases when a deleted edge decreases the noncommutative distance between two vertices $x$ and $y$, we show that the deleted edge need not be incident to either of the vertices $x$ or $y$, and in fact the deleted edge can be arbitrarily far from both in the graph-theoretic (number of edges) distance and in the weighted (geodesic) distance.  Moreover, deletion of a single vertex arbitrarily far from both $x$ and $y$ in either the graph-theoretic or weighted distance can change $d^D(x,y)$ by an arbitrarily large factor.
\end{abstract}

\maketitle

\section{Introduction}

Let $G=(V,E)$ be a finite simple graph on $V=\{1,\dots,n\}$, let
$\mathcal A\subset M_n(\mathbb C)$ be the algebra of diagonal matrices, and let
$D=(d_{ij})\in M_n(\mathbb R)$ be symmetric with nonnegative entries satisfying the
compatibility condition $i\ne j$, $\{i,j\}\notin E\Rightarrow d_{ij}=0$. Following
\cite{Con94} (see also \cite[(5)]{IKM01}, \cite[(8)]{Bes21}, and \cite[Definition 2.2]{CLPS25}), the
associated \textit{noncommutative distance} is
\begin{equation}\label{eq:def}
  \dD(i,j)=\sup\bigl\{\,|a_i-a_j| \;:\; A=\operatorname{diag}(a_1,\dots,a_n),\ \Norm{[A,D]}\le 1 \,\bigr\}.
\end{equation}
We write $w_{ij}=w(i,j):=d_{ij}^{-1}\in(0,+\infty]$ for the \textit{weight} of the edge
$\{i,j\}$, using the functional form $w(i,j)$ when the indices are cumbersome, and
$\ell^{D}$ for the geodesic (weighted) distance they induce. Note that $d^D \le \ell^D$; see \cite[Theorem 1]{Bes21}.  We say that the diagonal matrix $A$ is \textit{admissible} (for $D$) if $\Norm{[A,D]} \le 1$; note that we may assume that $A \in M_n(\R)$ (see, for instance, \cite[Section 2.3]{CLPS25}).

This paper is concerned with the effects of \textit{deleting an edge} on the noncommutative distance between two vertices.  In the noncommutative distance setting, deleting an edge $e=\{p,q\}$ means passing from $D$ to $D' = (d'_{ij})$, where $D'$ agrees with $D$ except that $d'_{pq}=d'_{qp}=0$.  It is known that deleting a vertex (weakly) increases $\dD$ \cite[Proposition 4(1)]{IKM01}, and it is clear that deleting an edge (weakly) increases the distance between two vertices in both the graph-theoretic and weighted metrics.  Besnard \cite[Section 1 and Section 3.1]{Bes21} asked the natural question of whether
\begin{equation}
  \label{eq:question}
  d^{D'}(x,y) \ge \dD(x,y)
\end{equation}
always holds, and the question was echoed by Clare, Li, Poon, and the author \cite[Remark 2.3]{CLPS25}.  The best result along these lines so far comes from combining \cite[Theorem 1, Theorem 5]{Bes21}: if $\Delta(H)$ denotes the maximum degree of the graph $H$, then
\[ d^{D'}(x,y) \ge \frac{\ell^{D'}(x,y)}{\Delta(G - e)} \ge \frac{\ell^{D}(x,y)}{\Delta(G - e)} \ge \frac{\dD(x,y)}{\Delta(G - e)}.\]

As it turns out, the noncommutative distance is not as well behaved as we might hope: indeed, there exist examples in which deletion of an edge \textit{decreases} the noncommutative distance between two vertices.

\begin{theorem}
\label{thm:main}
 Let $0 < \ep < 2/5$ and let $D$ be the Dirac operator of the $4$-cycle $1 - 2 - 3 - 4 - 1$ with entries
 \[ d_{12} = \ep, d_{23} = d_{34} = d_{41} = 1.\]
 Then,
 \[ \dD(1,2) \ge \kappa(\ep) \colonequals \frac{2}{\sqrt{5\ep^2 - 2\ep + 1}},\]
 while deleting the edge $\{1,2\}$ yields $d^{D'}(1,2) = 2 < \dD(1,2)$.  In particular, if $\ep = 1/5$, then $\dD(1,2) \ge \sqrt{5} > 2$.
\end{theorem}

Once it is known that there are such examples, we might ask precisely how large the discrepancy between $\dD$ and $d^{D'}$ can be.  Since we can always scale the weights by a common factor, what we are really asking is how large $\sup \dD/d^{D'}$ can be.  We first answer this question for $4$-cycles; along the way, we take the opportunity to correct the statement of \cite[Theorem 9]{IKM01}, the proof of which is fundamentally correct (see Proposition \ref{prop:ikm}).

\begin{theorem}
 \label{thm:sharp}
 Over all weighted $4$-cycles,
 \[ \dfrac{\dD(1,2)}{d^{D'}(1,2)} \le \frac{2}{\sqrt{3}},\]
and this upper bound is attained when $w_{12} = 3$, $w_{14} = w_{23} = 1/\sqrt{2}$, and $w_{34} = 1$, in which case $\dD(1,2) = \sqrt{3}$ and $d^{D'}(1,2) = 3/2$.
\end{theorem}

Moreover, we are able to prove that $2/\sqrt{3}$ is a tight upper bound for $\sup \dD/d^{D'}$ over \textit{all} graphs.

\begin{theorem}
 \label{thm:supbound}
 Let $D$ be a Dirac operator on a finite graph, let $e$ be any edge, and let $D'$ be obtained from $D$ by deleting $e$.  Then, for all vertices $x$ and $y$,
 \[ \dD(x,y) \le \frac{2}{\sqrt{3}}d^{D'}(x,y).\]
 Moreover, the constant $2/\sqrt{3}$ is sharp and is obtained by the weighted $4$-cycle of Theorem \ref{thm:sharp}, where $x$ and $y$ are the endpoints of the deleted edge.
\end{theorem}

We next study the effect of edge deletion on $n$-cycles for values of $n$ other than $4$.  We are able to relate admissibility of a matrix $A$ to a calculation involving the \textit{Pfaffians} of even-dimensional submatrices of the matrix $M = [A,D]$ (see Lemma \ref{lem:normMandXi}).  Since the Pfaffian is defined in terms of perfect matchings, this method is particularly effective for cycles, and we are able to completely characterize which $n$-cycles obey \eqref{eq:question}, the proof of which follows immediately from Theorems \ref{thm:odd}, \ref{thm:2mod4}, and \ref{thm:0mod4}.

\begin{theorem}
 \label{thm:cycles}
 Let $n \ge 3$ and let $D$ be a Dirac operator on the $n$-cycle.  For an edge $e$, define $D'$ to be the Dirac operator obtained by deleting the edge $e$.  If $4 \nmid n$, then
 \[ d^{D'}(x,y) \ge \dD(x,y)\]
 for all edges $e$ and for all choices $x,y$ of vertices.  On the other hand, if $4 \mid n$, then there exists a Dirac operator $D$, an edge $e$, and vertices $x, y$ such that
 \[ d^{D'}(x,y) < \dD(x,y).\]
\end{theorem}

In the examples above when edge deletion decreases the noncommutative distance between vertices $x$ and $y$, the deleted edge $e$ is always $\{x,y\}$.  This leads to the natural question as to whether the distance of the deleted edge from the vertices $x$ and $y$ plays a role; along these lines, Besnard speculates that a vertex $z$ that is geodesically very far from $x$ and $y$ should not matter much, if at all, for the computation of $\dD(x,y)$ \cite[Section 5]{Bes21}.  However, as shown in Section \ref{sect:far}, this is not the case; the following result is an immediate consequence of Corollary \ref{cor:far}, Theorem \ref{thm:nonlocal}, and Theorem \ref{thm:farweight}.

\begin{theorem}
 \label{thm:farsummary}
 In each of the following, $G$ is a graph with Dirac operator $D$, $x,y,z$ are vertices, $e$ is an edge, $D'$ denotes the Dirac operator of $G - e$, and $D''$ denotes the Dirac operator of $G - z$.
 \begin{itemize}
  \item[(i)] For every positive integer $N$, there exists a choice of $G$, $D$, $x$, $y$, and $e$ such that the graph-theoretic distance from $e$ to each of $x$ and $y$ is at least $N$ and $d^{D'}(x,y) < \dD(x,y)$.
  \item[(ii)] For every positive integer $N > 1$ and $\eta > 0$, there exists a choice of $G$, $D$, $x$, $y$, and $z$ such that the distance from $z$ to each of $x$ and $y$ is $N$, $\dD(x,y) < \eta$, and $d^{D''}(x,y) = 1$.
  \item[(iii)] For every $\rho > 0$, there exists a choice of $G$, $D$, $x$, $y$, and $e$ such that the weighted distance from $e$ to each of $x$ and $y$ is $\rho$ and
  \[ d^{D'}(x,y) \le \frac{11}{5} < \sqrt{5} \le \dD(x,y).\]
 \end{itemize}
\end{theorem}

We also note that Theorem \ref{thm:farsummary} (i), (iii) can be combined so that the graph-theoretic distance and the weighted distance are both arbitrarily large while $d^{D'}(x,y) < \dD(x,y)$; see Remark \ref{rem:combined}.

This paper is organized as follows.  Section \ref{sect:prelim} is dedicated to the preliminary results about noncommutative distances and skew-symmetric matrices that are needed later in the paper.  Section \ref{sect:counterex} is dedicated to the proof of Theorem \ref{thm:main}, which provides a concrete example on a $4$-cycle showing that the deletion of an edge can decrease the noncommutative distance between vertices.  In Section \ref{sect:4cycle}, we analyze how edge deletion affects noncommutative distances on the $4$-cycle and prove Theorem \ref{thm:sharp}, and in Section \ref{sect:sharp}, we prove Theorem \ref{thm:supbound}, showing that the bound from Theorem \ref{thm:sharp} is in fact tight for all graphs.  In Section \ref{sect:cycles}, we analyze how edge deletion affects noncommutative distances in an $n$-cycle and prove Theorem \ref{thm:cycles}.  Finally, in Section \ref{sect:far} we prove Theorem \ref{thm:farsummary}, providing constructions showing that deletion of an edge or vertex can materially affect the noncommutative distance between two far away vertices.


\section{Preliminaries}
\label{sect:prelim}

In this section we collect a few results that are used later in the paper, starting with useful results regarding the noncommutative distance $\dD$.  First, it is known that $d^D$ is bounded by the weighted (geodesic) distance $\ell^D$.

\begin{proposition}[{\cite[Theorem 1, Corollary 1]{Bes21}}, {\cite[Proposition 4]{IKM01}}]
\label{prop:geod}
For all vertices $x,y$, $\dD(x,y) \le \ell^D(x,y)$, and $\dD(x,y) = \ell^D(x,y)$ when $\{x,y\}$ is a bridge.
\end{proposition}

Furthermore, as previously mentioned, it is known that deletion of a vertex only increases the noncommutative distance.

\begin{proposition}[{\cite[Proposition 4(1)]{IKM01}}, {\cite[Lemma 2, Equation (9)]{Bes21}}]
\label{prop:deletion}
If $D'$ is obtained from $D$ by setting to $0$ the $i^{\text{th}}$ row and column (which corresponds to the deletion of the vertex $i$), then $d^{D'} \ge \dD$.  More generally, if $V'$ is a subset of vertices of $V$ and $d^{D}_{V'}$ is the associated noncommutative distance on the induced subgraph, then $\dD_{V'} \ge \dD$.
\end{proposition}

For vertices $x$ and $y$ in $V$, we denote  by $\PP(x,y)$ the subgraph of $G$ induced on the union of all simple paths joining $x$ and $y$.  The noncommutative distance between $x$ and $y$ depends only on $\PP(x,y)$.

\begin{proposition}[pruning; {\cite[Proposition 4(2)]{IKM01}}, {\cite[Theorems 3,4]{Bes21}}]
\label{prop:prune}
For $x \neq y$, $\dD(x,y) = d^{D_\PP}(x,y)$, where $D_\PP$ is induced on $\PP(x,y)$.
\end{proposition}

The noncommutative distance between the endpoints of a path with four vertices (also called a \textit{chain} with three edges) is known; for longer paths, the situation is difficult; see \cite{CLPS25}.

\begin{proposition}[{\cite[Proposition 3]{Bes21}}, {\cite[Corollary 4.2]{CLPS25}}, {\cite{Atz96, BLS94}}]
\label{prop:chain}
 Let $G$ be the path with four vertices, $1 - 2 - 3 - 4$, such that $w_{ij}$ denotes the weight of the edge $\{i,j\}$.  Then,
 \[ \dD(1,4) = \begin{cases}
                \frac{\sqrt{w_{12}^2 + w_{23}^2}\sqrt{w_{34}^2 + w_{23}^2}}{w_{23}}, & w_{23} > \sqrt{w_{12}w_{34}},\\
                w_{12} + w_{34}, & \text{ otherwise.}
               \end{cases}\]
In particular, if $w_{12} = w_{23} = w_{34} = 1$, then $\dD(1,4) = 2$.
\end{proposition}

We now collect some facts about matrices.  For $S \subseteq \{1, \dots, n\}$, we denote by $M_S$ the principal submatrix of $M$ on $S$; the matrix $M_S$ is sometimes referred to as the \textit{compression} of $M$ to $S$.  The following is a well-known result regarding matrix norms.

\begin{lemma}\cite[Corollary 3.1.3]{HJ94}
 \label{lem:compression}
 For $S \subseteq \{1, \dots, n\}$, $\Norm{M_S} \le \Norm{M}.$
\end{lemma}

Finally, we recall some facts about real skew-symmetric matrices; these are particularly useful in our calculations since $[A,D]$ is skew-symmetric.

For a skew-symmetric matrix $N$ of even size $2k$ and indexed by an ordered set $S$, we define the \textit{Pfaffian} of the matrix $N$ by
 \begin{equation}
 \label{eq:Pf}
 \Pf(N) = \sum_{F} \sgn(F) \prod_{\{i,j\} \in F} n_{ij},
 \end{equation}
 where the sum is over all perfect matchings
 \[ F = \{ \{i_1, j_1\}, \dots \{i_k, j_k\}\}\]
 of the set $S$ and $\sgn(F)$ is the sign of the permutation sending
 \[ x \mapsto \begin{cases}
               i_{(x+1)/2}, &\text{ if } x \text{ is odd},\\
               j_{x/2}, &\text{ if } x \text{ is even};
              \end{cases}
\]
see \cite[pp. 317--318]{LP09}.

We need the following technical lemma regarding the Pfaffian of our weighted $n$-cycle, where $n$ is even.

\begin{lemma}
 \label{lem:Pf4}
 Let $n \ge 4$ be even and suppose $M = (m_{i,j}) = [A,D]$, where $A$ is a diagonal matrix and $D$ is the Dirac operator for an $n$-cycle.  Then,
 \[ \Pf(M) = m_{12}m_{34} \cdots m_{n-1,n} - m_{23}m_{45} \cdots m_{n,1}.\]
\end{lemma}

\begin{proof}
 Recall the definition of the Pfaffian \eqref{eq:Pf}.  For $M$, only the two perfect matchings of the cycle contribute, and the sign of each term is a straightforward calculation.
\end{proof}

The results collected in the following proposition follow from the canonical form for such matrices (see, e.g., \cite[Theorem 5]{KL97}), basic results about Pfaffians (see, e.g., \cite[Chapter 7]{God93}), and an examination of the characteristic polynomial of $M$.

\begin{proposition}
 \label{prop:skew}
 Let $M \in M_{2r}(\R)$ be a skew-symmetric matrix.
 \begin{itemize}
  \item[(i)] The nonzero eigenvalues of $M$ occur in pairs $\pm is_j$, where each $s_j$ is a singular value of $M$; in particular, each singular value $s_j$ has even multiplicity.
  \item[(ii)] Let $M = (m_{ij}) \in M_4(\R)$, $F = F(M) = \sum_{i<j} m_{ij}^2$, and $\Pf = \Pf(M) = m_{12}m_{34} - m_{13}m_{24} + m_{14}m_{23}$.  Then, the following hold:
    \begin{itemize}
      \item[\textbullet] $\det M = \Pf(M)^2$,
      \item[\textbullet] the two singular values $s_1 \ge s_2$ each have multiplicity two and satisfy $s_1^2 + s_2^2 = F$ and $s_1s_2 = |\Pf|$,
      \item[\textbullet] the eigenvalues of $iM$ are $\pm \frac{1}{\sqrt{2}}\left(F \pm \sqrt{F^2 - 4\Pf^2} \right)^{1/2}$,
      \item[\textbullet] $\Norm{M} \le 1 \iff F \le 2 \text{ and } F \le 1 + \Pf^2.$
    \end{itemize}
 \end{itemize}
\end{proposition}

\begin{remark}
 The formula for the Pfaffian from Proposition \ref{prop:skew} agrees with the formula from Lemma \ref{lem:Pf4} when we note that $m_{14} = -m_{41}$ and $m_{13} = m_{24} = 0$.
\end{remark}


\section{An example of edge deletion decreasing the noncommutative distance between two vertices}
\label{sect:counterex}

In this section, we present a simple example of a weighted 4-cycle showing that deleting an edge can in fact decrease the noncommutative distance between two vertices.

\begin{proof}[Proof of Theorem \ref{thm:main}]
 Fix $0 < \ep < 2/5$, let $\kappa = \kappa(\ep)$, and let
 \[ A = \diag\left( \kappa, 0, \kappa\frac{1 - \ep}{2}, \kappa\frac{1 + \ep}{2}\right).\]
 If $M = (m_{ij}) = [A,D]$, then $m_{ij} = (a_i - a_j)d_{ij}$, $M$ is skew-symmetric, and
 \[ M = \begin{pmatrix}
         0          & \kappa\ep & 0      & \kappa\frac{1 - \ep}{2}\\
         -\kappa\ep & 0         & -\kappa\frac{1 - \ep}{2} & 0\\
         0          & \kappa\frac{1 - \ep}{2} & 0 & -\kappa\ep\\
         -\kappa\frac{1 - \ep}{2} & 0 & \kappa\ep & 0\\
        \end{pmatrix}.
\]
Moreover,
\[MM^t = \left((\kappa\ep)^2 + \left(\kappa\frac{1 - \ep}{2} \right)^2\right)I = I,\]
so $\Norm{[A,D]} = \Norm{M} = 1$, $A$ is admissible in \eqref{eq:def}, and $\dD(1,2) \ge a_1 - a_2 = \kappa$.

On the other hand, deleting the edge $\{1,2\}$ leaves a path on four vertices with endpoints $1$, $2$ and all edges having weight $1$.  By Propositions \ref{prop:prune} and \ref{prop:chain}, $d^{D'}(1,2) = 2$.

Since $0 < \ep < 2/5$, $\ep(5\ep - 2) < 0$, and so $\kappa = \kappa(\ep) > 2$, completing the proof that $\dD(1,2) > d^{D'}(1,2)$.  Finally, since $5\ep^2 - 2\ep + 1$ is minimized at $\ep = 1/5$, the largest value of $\kappa$ is $\sqrt{5}$.
\end{proof}

\begin{example}
 In terms of weights, the weighted $4$-cycle with $d_{12} = 1/5$ and $d_{23} = d_{34} = d_{41} = 1$ corresponds to a cycle whose edges have weights $5$, $1$, $1$, $1$.  See Figure \ref{fig:main}.

 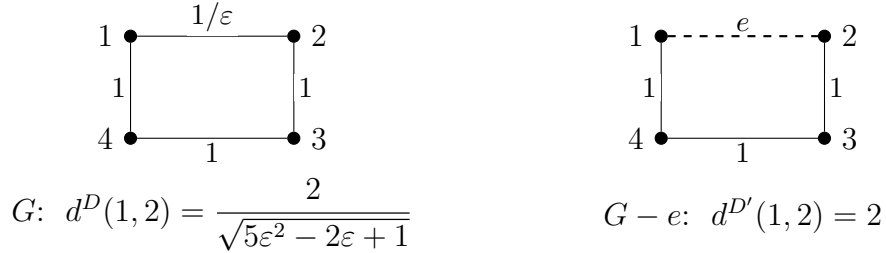
\begin{figure}[h]
\centering
\begin{tikzpicture}[scale=1.35]
  \node[vtx,label=left:{$1$}]  (a) at (0,1) {};
  \node[vtx,label=right:{$2$}] (b) at (1.6,1) {};
  \node[vtx,label=right:{$3$}] (c) at (1.6,0) {};
  \node[vtx,label=left:{$4$}]  (d) at (0,0) {};
  \draw (a)--(b) node[wt,midway,above]{$1/\varepsilon$};
  \draw (b)--(c) node[wt,midway,right]{$1$};
  \draw (c)--(d) node[wt,midway,below]{$1$};
  \draw (d)--(a) node[wt,midway,left]{$1$};
  \node at (0.8,-0.75) {$G$:\ \ $d^D(1,2)=\dfrac{2}{\sqrt{5\varepsilon^{2}-2\varepsilon+1}}$};
  \begin{scope}[xshift=5.2cm]
  \node[vtx,label=left:{$1$}]  (a) at (0,1) {};
  \node[vtx,label=right:{$2$}] (b) at (1.6,1) {};
  \node[vtx,label=right:{$3$}] (c) at (1.6,0) {};
  \node[vtx,label=left:{$4$}]  (d) at (0,0) {};
  \draw[del] (a)--(b) node[wt,midway,above]{$e$};
  \draw (b)--(c) node[wt,midway,right]{$1$};
  \draw (c)--(d) node[wt,midway,below]{$1$};
  \draw (d)--(a) node[wt,midway,left]{$1$};
  \node at (0.8,-0.75) {$G-e$:\ \ $d^{D'}(1,2)=2$};
  \end{scope}
\end{tikzpicture}
\medskip
\caption{Theorem~\ref{thm:main}. Edges are labeled by their weights $w_{ij}=d_{ij}^{-1}$. For
$\varepsilon<\tfrac25$, i.e.\ $w(1,2)>\tfrac52$, deleting $e$ \textit{decreases} the
distance; at $\varepsilon=\tfrac15$ the values are $\sqrt5$ and $2$.}
\label{fig:main}
\end{figure}
 \end{example}

\begin{remark}
 Although explicit formulas are known for the noncommutative distance between two vertices in weighted 4-cycles \cite[Theorem 9]{IKM01}, there are small typos in the stated result.  We provide a corrected version (see Proposition \ref{prop:ikm}) and use it to give a complete analysis of the weighted 4-cycle in Section \ref{sect:4cycle}, but the proof of Theorem \ref{thm:main} does not rely on these formulas for a weighted 4-cycle in order to keep it self-contained.  By Proposition \ref{prop:ikm}, we have $\dD(1,2) = \kappa$ in Theorem \ref{thm:main}.
\end{remark}


\section{Full analysis of edge deletion in a weighted 4-cycle}
\label{sect:4cycle}

In this section, $G$ is the $4$-cycle $1 - 2 - 3 - 4 - 1$ whose Dirac operator $D$ has entries $d_{13} = d_{24} = 0$ while $d_{12}, d_{23}, d_{34}, d_{14} > 0$; the edge $e = \{1,2\}$ is the edge to be deleted; and statements are phrased in terms of the weights $w_{ij} = d_{ij}^{-1}$ (see Figure \ref{fig:c4}).

\begin{figure}[h]
\centering
\begin{tikzpicture}[scale=1.5]
  \node[vtx,label=left:{$1$}]  (v1) at (0,1) {};
  \node[vtx,label=right:{$2$}] (v2) at (1.9,1) {};
  \node[vtx,label=right:{$3$}] (v3) at (1.9,0) {};
  \node[vtx,label=left:{$4$}]  (v4) at (0,0) {};
  \draw[del] (v1)--(v2) node[wt,midway,above]{$w_{12}$};
  \draw (v2)--(v3) node[wt,midway,right]{$w_{23}$};
  \draw (v3)--(v4) node[wt,midway,below]{$w_{34}$};
  \draw (v4)--(v1) node[wt,midway,left]{$w_{14}$};
\end{tikzpicture}
\medskip
\caption{Notation of \S4: the dashed edge, of weight $w_{12}$, is the one to be deleted,
and $d^D(1,2)$ is the quantity of interest. Deleting it leaves a path on four vertices with weights
$w_{14}-w_{34}-w_{23}$.}
\label{fig:c4}
\end{figure}
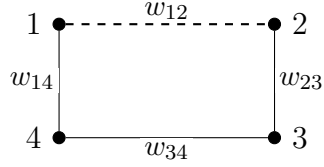

By Propositions \ref{prop:prune} and \ref{prop:chain}, deleting $e$ leaves a path on four vertices joining $1$ and $2$ whose edges (going from $1$ to $2$) have weights $w_{14}$, $w_{34}$, and $w_{23}$, so
\begin{equation}
 \label{eq:4path}
 d^{D'}(1,2) = \begin{cases}
                \frac{\sqrt{w_{14}^2 + w_{34}^2}\sqrt{w_{23}^2 + w_{34}^2}}{w_{34}}, & w_{34} > \sqrt{w_{14}w_{23}},\\
                w_{14} + w_{23}, & \text{ otherwise.}
               \end{cases}
\end{equation}

\subsection{The formula for the noncommutative distance between two vertices in a 4-cycle}

We now take the opportunity to correct the statement of \cite[Theorem 9]{IKM01}.  The analysis there is fundamentally correct, and this ``correction'' largely amounts to correcting typos in the statement.

\begin{proposition}[{\cite[Theorem 9]{IKM01}, corrected}]
 \label{prop:ikm}
 Using the notation above, let
 \[ C \colonequals \left((w_{14} + w_{23})^2w_{34} +(w_{12} - w_{34})(w_{14}w_{23} - w_{34}^2) \right)\left((w_{14} - w_{23})^2w_{34} +(w_{12} + w_{34})(w_{14}w_{23} + w_{34}^2) \right).\]
 Then,
 \[ \dD(1,2) = \begin{cases}
                w_{12}, & w_{12} \le w_{34},\\
                \dfrac{w_{12}\sqrt{w_{14}^2 + w_{34}^2}\sqrt{w_{23}^2 + w_{34}^2}}{|w_{14}w_{23} - w_{12}w_{34}|}, & w_{12} > w_{34} \text{ and } C \le 0, \\
                \dfrac{w_{12}(w_{14} + w_{23})}{\sqrt{(w_{14} + w_{23})^2 + (w_{12} - w_{34})^2}}, & \text{ otherwise,}
               \end{cases}
\]
and
\[ \dD(1,3) = \begin{cases}
               \sqrt{w_{14}^2 + w_{34}^2}, & (w_{14}^2 + w_{34}^2)^2 \le (w_{12}w_{34} - w_{14}w_{23})^2,\\
               \sqrt{w_{12}^2 + w_{23}^2}, & (w_{12}^2 + w_{23}^2)^2 \le (w_{12}w_{34} - w_{14}w_{23})^2,\\
               \dfrac{w_{12}w_{14} + w_{23}w_{34}}{\sqrt{(w_{14} + w_{23})^2 + (w_{12} - w_{34})^2}}, & (w_{14}^2 + w_{34}^2)^2, (w_{12}^2 + w_{23}^2)^2 > (w_{12}w_{34} - w_{14}w_{23})^2\\
                     & \text{ and } w_{12}w_{34} \ge w_{14}w_{23},\\
               \dfrac{w_{12}w_{14} + w_{23}w_{34}}{\sqrt{(w_{14} - w_{23})^2 + (w_{12} + w_{34})^2}}, & \text{ otherwise.}
              \end{cases}
\]
\end{proposition}

\begin{remark}
 \label{rem:ikm}
 \begin{itemize}
  \item[(1)] Setting $d_1 = w_{12}$, $d_2 = w_{13} = \infty$, $d_3 = w_{14}$, $d_4 = w_{23}$, $d_5 = w_{24} = \infty$, and $d_6 = w_{34}$ recovers the notation of \cite[Theorem 9]{IKM01}.
  \item[(2)] The hypothesis is stated there as ``$1/d_2 = 1/d_5 = \infty$'' and in the proof as ``$d_2 = d_5 = 0$''; the hypothesis actually used is ``$d_2 = d_5 = \infty$''.
  \item[(3)] There is a simplification to the third branch of $\dD(1,2)$ from the printed version.  Let
  \[ x_2 = \frac{w_{12}(w_{14} + w_{23})}{\sqrt{(w_{14} + w_{23})^2 + (w_{12} - w_{34})^2}}, \qquad x_3 = \frac{w_{12}|w_{14} - w_{23}|}{\sqrt{(w_{14} - w_{23})^2} + (w_{12} + w_{34})^2};\]
  in the original statement, the third branch was $\max(x_2, x_3)$, and it turns out $x_2 \ge x_3$ always.  When comparing $x_2$ to $x_3$, we can cancel the common factor $w_{12}$, and after some simplifying,
  we find that
  \[ x_2 \ge x_3 \iff \frac{1}{\sqrt{1 + \left(\dfrac{w_{12} - w_{34}}{w_{14} + w_{23}}\right)^2}} \ge \frac{1}{\sqrt{1 + \left(\dfrac{w_{12} + w_{34}}{w_{14} - w_{23}} \right)^2}}.\]
  Noting that the function $t \mapsto 1/\sqrt{1 + (c/t)^2}$ is strictly decreasing as a function of $c/t$, this implies that
  \[ x_2 \ge x_3 \iff \frac{|w_{12} - w_{34}|}{w_{14} + w_{23}} \le \frac{w_{12} + w_{34}}{|w_{14} - w_{23}|}.\]
  This further implies
  \[ x_2 \ge x_3 \iff (w_{14} + w_{23})(w_{12} + w_{34}) \ge |w_{14} - w_{23}| \cdot |w_{12} - w_{34}|.\]
  This last inequality is immediate from the triangle inequality since all $w_{ij}$ are nonnegative.  Hence $x_2 \ge x_3$.
  \item[(4)] The separate ``$d_1d_6 = d_3d_4$'' branch from $\dD(1,2)$ can be dropped: in that case, the first factor of $C$ equals $(d_3 + d_4)^2d_6 + d_6(d_1 - d_6)^2 > 0$ and the second is also positive, so $C > 0$ and the third branch applies (and, in fact, the formula from the ``$d_1d_6 = d_3d_4$'' branch coincides with our formula above when $C > 0$, subject to the assumption that $d_1d_6 = d_3d_4$).
  \item[(5)] The numerator in the third branch of $\dD(1,2)$ was printed in the statement as $d_1\sqrt{d_3^2 + d_4^2}$; however, the correct value printed above is what was used in the proof.
  \item[(6)] The left sides of the inequalities in the first two branches of $\dD(1,3)$ need to be squared; again, this is what was used in the original proof.
  \item[(7)] The result printed above was verified numerically using a semidefinite-program computation by Claude Opus 5 over 1000 random weighted $4$-cycles, which included at least 100 examples corresponding to each branch in each formula.
 \end{itemize}
\end{remark}

\subsection{Analysis}

In order to simplify the calculations, we introduce some notation.  First, recall the definition of $C$ from Proposition \ref{prop:ikm}.  We will define
\[ \Lambda \colonequals \frac{\sqrt{w_{14}^2 + w_{34}^2}\sqrt{w_{23}^2 + w_{34}^2}}{w_{34}},\]
\[ \sigma \colonequals w_{14} + w_{23}, \]
and $\lambda \colonequals d^{D'}(1,2)$, so that
\[ \lambda = \begin{cases}
                \Lambda, & w_{34} > \sqrt{w_{14}w_{23}},\\
                \sigma, & \text{ otherwise.}
               \end{cases}\]
Note that $\Lambda \ge \lambda \ge \sigma$, since $\Lambda \ge \sigma$ is equivalent to $(w_{14}w_{23} - w_{34}^2)^2 \ge 0$ after clearing denominators.

\begin{lemma}
 \label{lem:notbranch1}
 Using the notation above, if $\dD(1,2) > d^{D'}(1,2)$, then $w_{12} > w_{34}$.
\end{lemma}

\begin{proof}
 Assume $w_{12} \le w_{34}$.  By Proposition \ref{prop:ikm}, $\dD(1,2) = w_{12} \le w_{34}$.  Based on the formula for $d^{D'}(1,2) = \lambda$, we divide into cases based on whether $w_{34}$ is greater than or less than or equal to $\sqrt{w_{14}w_{23}}$.  If $w_{34} \le \sqrt{w_{14}w_{23}}$, then
 \[ \dD(1,2) \le \sqrt{w_{14}w_{23}} \le \frac{w_{14} + w_{23}}{2} < \sigma = \lambda. \]
 On the other hand, if $w_{34} > \sqrt{w_{14}w_{23}}$, then $\sqrt{w_{14}^2 + w_{34}^2} \sqrt{w_{23}^2 + w_{34}^2} \ge w_{34}^2$ implies
 \[ \dD(1,2) \le w_{34} \le \Lambda = \lambda.\]  The result follows.
\end{proof}

Proposition \ref{prop:ikm} and Lemma \ref{lem:notbranch1} show that the value of $\dD(1,2)$ depends on whether or not $C$ is positive.  The following lemma is useful for working with $C$.

\begin{lemma}
 \label{lem:C}
 Assume $w_{12} > w_{34}$.  Using the notation above, $C \le 0$ if and only if $w_{34} > \sqrt{w_{14}w_{23}}$ and
 \[ w_{12} \ge w_{34} + \frac{\sigma^2 w_{34}}{w_{34}^2 - w_{14}w_{23}}.\]
 Moreover, if $C \le 0$, then $w_{34} > \sqrt{w_{14}w_{23}}$ and $w_{12}w_{34} > w_{14}w_{23}$.
\end{lemma}

\begin{proof}
 Since $(w_{14} - w_{23})^2w_{34} +(w_{12} + w_{34})(w_{14}w_{23} + w_{34}^2)$ is always positive for positive $w_{ij}$, we have that $C \le 0$ if and only if $(w_{14} + w_{23})^2w_{34} +(w_{12} - w_{34})(w_{14}w_{23} - w_{34}^2) \le 0.$  If $w_{34}^2 \le w_{14}w_{23}$, then the first factor of $C$ is the sum of a nonnegative term and a positive term, and hence $C > 0$.  When $w_{34}^2 > w_{14}w_{23}$, we rearrange the first inequality; this proves the first claim.

 Next, since $w_{12} > w_{34}$, $w_{34} \le \sqrt{w_{14}w_{23}}$ implies $C > 0$, so, if $C \le 0$, then $w_{34} > \sqrt{w_{14}w_{23}}$.  Finally, if $C \le 0$ and $w_{12} > w_{34}$, we have
 \[ w_{14}w_{23} < w_{34}^2 < w_{12}w_{34}.\]
\end{proof}

We are now able to prove one of our two main results of this section.

\begin{theorem}
 \label{thm:alwaysaw12}
 Using the notation above, for any choice of $w_{14}, w_{23}, w_{34}$, $\dD(1,2) > d^{D'}(1,2)$ for all $w_{12}$ sufficiently large.  In particular, if $w_{34} > \sqrt{w_{14}w_{23}}$, then $\dD(1,2) > d^{D'}(1,2)$ whenever $w_{12} \ge w_{34} + \sigma^2 w_{34}/(w_{34}^2 - w_{14}w_{23})$, and if $w_{34} \le \sqrt{w_{14}w_{23}}$, then $\dD(1,2) > d^{D'}(1,2)$ whenever $w_{12} > (\sigma^2 + w_{34}^2)/(2w_{34}).$
\end{theorem}

\begin{proof}
 Assume first $w_{34} > \sqrt{w_{14}w_{23}}$ and $w_{12} \ge w_{34} + \sigma^2 w_{34}/(w_{34}^2 - w_{14}w_{23})$.  This implies $d^{D'}(1,2) = \lambda = \Lambda$ and so
 \begin{align*}
  \dD(1,2) &= \dfrac{w_{12}\sqrt{w_{14}^2 + w_{34}^2}\sqrt{w_{23}^2 + w_{34}^2}}{|w_{14}w_{23} - w_{12}w_{34}|} && \text{Proposition \ref{prop:ikm} and Lemma \ref{lem:C}}\\
           &= \dfrac{w_{12}w_{34}\Lambda}{|w_{14}w_{23} - w_{12}w_{34}|} && \\
           &= \dfrac{w_{12}w_{34}\Lambda}{w_{12}w_{34} - w_{14}w_{23}} && \text{Lemma \ref{lem:C}}\\
           &= \dfrac{\Lambda}{1 - \frac{w_{14}w_{23}}{w_{12}w_{34}}} && \\
           &> \Lambda && \text{Lemma \ref{lem:C}}\\
           &= d^{D'}(1,2).
 \end{align*}

Next, assume $w_{34} \le \sqrt{w_{14}w_{23}}$ and $w_{12} > (\sigma^2 + w_{34}^2)/(2w_{34}).$  Since $\sigma^2 \ge 4w_{14}w_{23} \ge 4w_{34}^2$, we have
\[w_{12} > \frac{\sigma^2 + w_{34}^2}{2w_{34}} \ge \frac{5w_{34}}{2} > w_{34}.\]  Further, since $w_{34} \le \sqrt{w_{14}w_{23}}$, we have $d^{D'}(1,2) = \lambda = \sigma$ and, by Lemma \ref{lem:C}, $C > 0$, so
\[ \dD(1,2) = \dfrac{w_{12}(w_{14} + w_{23})}{\sqrt{(w_{14} + w_{23})^2 + (w_{12} - w_{34})^2}}.\]  Since $w_{12} > (\sigma^2 + w_{34}^2)/(2w_{34})$ and $\lambda = \sigma$, we have
\[ (\sigma^2 - \lambda^2)w_{12}^2 + 2\lambda^2w_{34}w_{12} - \lambda^2(\sigma^2 + w_{34}^2) = 2\lambda^2w_{34}w_{12} - \lambda^2(\sigma^2 + w_{34}^2) > 0,\]
which is equivalent to
\[ \dD(1,2) = \dfrac{w_{12}(w_{14} + w_{23})}{\sqrt{(w_{14} + w_{23})^2 + (w_{12} - w_{34})^2}} = \dfrac{w_{12}\sigma}{\sqrt{\sigma^2 + (w_{12} - w_{34})^2}} > \lambda = d^{D'}(1,2).\]
The result follows.
\end{proof}

Our next goal is to determine how much the noncommutative distance between $1$ and $2$ can decrease after deletion of the edge $\{1,2\}$.  Of course, we could always scale the weights of the edges, so what we are really asking for is an upper bound for $\dD(1,2)/d^{D'}(1,2)$.  With this in mind, we first provide an upper bound for $\dD(1,2)$.

\begin{proposition}
 \label{prop:dD12bound}
 Using the notation above, for fixed $w_{14}, w_{23}, w_{34}$, $\dD(1,2) \le \sqrt{\sigma^2 + w_{34}^2}$, and this upper bound is attained when $w_{12} = (\sigma^2 + w_{34}^2)/w_{34}$.
\end{proposition}

\begin{proof}
 We prove first that each piecewise branch of the formula for $\dD(1,2)$ from Proposition \ref{prop:ikm} is bounded above by $\sqrt{\sigma^2 + w_{34}^2}$.

 Assume first that $w_{12} \le w_{34}$.  Then,
 \[ \dD(1,2) = w_{12} \le w_{34} < \sqrt{\sigma^2 + w_{34}^2}.\]

 Next, assume $w_{12} > w_{34}$ and $C \le 0$.  As in the proof of Theorem \ref{thm:alwaysaw12}, we see that
 \[ \dD(1,2) = \dfrac{\Lambda}{1 - \frac{w_{14}w_{23}}{w_{12}w_{34}}},\]
 and this function is decreasing in $w_{12}$, and, by Lemma \ref{lem:C}, the minimum value of $w_{12}$ is
 \[ w_{12} = w_{34} + \frac{\sigma^2 w_{34}}{w_{34}^2 - w_{14}w_{23}}.\]
 Using this value for $w_{12}$ and noting that
 \[ \dfrac{w_{14}w_{23}(w_{14} + w_{23})^2(w_{14}^2 + w_{34}^2)(w_{23}^2 + w_{34}^2)}{(w_{34}^2 - w_{14}w_{23})^2} > 0,\]
 after rearrangement we have
 \[ (2w_{34}^2 - w_{14}w_{23})w_{12}^2 - 2w_{34}((w_{14} + w_{23})^2 + w_{34}^2)w_{12} + w_{14}w_{23}((w_{14} + w_{23})^2 + w_{34}^2) > 0.\]
 After multiplying through by $w_{14}w_{23}$ and rearranging, we have
 \[ \dD(1,2)^2 = \dfrac{w_{12}^2(w_{14}^2 + w_{34}^2)(w_{23}^2 + w_{34}^2)}{(w_{12}w_{34} - w_{14}w_{23})^2} < (w_{14} + w_{23})^2 + w_{34}^2 = \left(\sqrt{\sigma^2 + w_{34}^2}\right)^2,\] which shows $\dD(1,2) \le \sqrt{\sigma^2 + w_{34}^2}$ for this branch as well.

 For the final branch, we assume $w_{12} > w_{34}$ and $C > 0$.  In this case, we note that
 \begin{equation}
 \label{eq:supd12}
  (\sigma^2 + w_{34}^2)(\sigma^2 + (w_{12} - w_{34})^2) - w_{12}^2\sigma^2 = (\sigma^2 + w_{34}^2 - w_{12}w_{34})^2 \ge 0
 \end{equation}
 (with equality precisely when $w_{12} = (\sigma^2 + w_{34}^2)/w_{34}$), and this inequality rearranges to give us
 \[ \dD(1,2)^2 = \dfrac{w_{12}^2\sigma^2}{\sigma^2 + (w_{12} - w_{34})^2} \le \sigma^2 + w_{34}^2 = \left(\sqrt{\sigma^2 + w_{34}^2}\right)^2,\]
 which shows that $\dD(1,2) \le \sqrt{\sigma^2 + w_{34}^2}$ in any case.

 Finally, assume $w_{12} = (\sigma^2 + w_{34}^2)/w_{34}$.  Note that $w_{12} = (\sigma^2 + w_{34}^2)/w_{34} > w_{34}$.  If $w_{34} \le \sqrt{w_{14}w_{23}}$, then $C > 0$ by Lemma \ref{lem:C}.  Otherwise, $w_{34}^2 > w_{14}w_{23}$, and, since
 \[ \dfrac{\sigma^2 + w_{34}^2}{w_{34}} - w_{34} - \dfrac{\sigma^2w_{34}}{w_{34}^2 - w_{14}w_{23}} = \dfrac{-w_{14}w_{23}\sigma^2}{w_{34}(w_{34}^2 - w_{14}w_{23})} < 0,\]
 by Lemma \ref{lem:C} we again have $C > 0$. In either case, by Proposition \ref{prop:ikm}
 \[ \dD(1,2) = \dfrac{w_{12}\sigma}{\sqrt{\sigma^2 + (w_{12} - w_{34})^2}}.\]  By \eqref{eq:supd12} above, this shows $\dD(1,2) = \sqrt{\sigma^2 + w_{34}^2}$ when $w_{12} = (\sigma^2 + w_{34}^2)/w_{34}$, as desired.
\end{proof}

We are now able to prove that an upper bound for $\dD(1,2)/d^{D'}(1,2)$ exists and can be attained.

\begin{proof}[Proof of Theorem \ref{thm:sharp}]
 We once more use the notation introduced above.  By Proposition \ref{prop:dD12bound},
 \[ \dfrac{\dD(1,2)}{d^{D'}(1,2)} \le \dfrac{\sqrt{\sigma^2 + w_{34}^2}}{\lambda}.\]
 Since $\sqrt{\sigma^2 + w_{34}^2}/\lambda$ is invariant under simultaneous scaling of all edge weights, we set $w_{34} = 1$ and define $\rho:= w_{14}w_{23}$.  Note that by Lemma \ref{lem:notbranch1}, $w_{12} > 1$, and, by \eqref{eq:4path}, we have two cases: $\rho \ge 1$ and $\rho < 1$.  Set $R \colonequals \sqrt{\sigma^2 + 1}/\lambda$.

 Assume first that $\rho \ge 1$.  By \eqref{eq:4path}, $\lambda = \sigma$, and
 \[ R = \dfrac{\sqrt{\sigma^2 + 1}}{\sigma} = \sqrt{1 + \dfrac{1}{\sigma^2}}.\]
 Since $R$ is decreasing in $\sigma$ and
 \[\sigma = w_{14} + w_{23} \ge 2\sqrt{w_{14}w_{23}} = 2\sqrt{\rho} \ge 2,\]
 we have
 \[ R \le \dfrac{2^2 + 1}{2} = \dfrac{\sqrt{5}}{2} < \dfrac{2}{\sqrt{3}}.\]

 Now assume $\rho < 1$.  By \eqref{eq:4path}, we have
 \[ \lambda = \Lambda = \sqrt{(w_{14}^2 + 1)(w_{23}^2 + 1)} = \sqrt{(1 - \rho)^2 + \sigma^2},\]
 and so $R = \dfrac{\sqrt{\sigma^2 + 1}}{\sqrt{(1 - \rho)^2 + \sigma^2}}$; note that for a fixed $\sigma$ that $R$ increases with $\rho$.  Since
 \[ \rho = w_{14}w_{23} \le \dfrac{(w_{14} + w_{23})^2}{4} = \dfrac{\sigma^2}{4},\] we may set $\rho = \min(1, \sigma^2/4)$.  When $\sigma \ge 2$, as $\rho \to 1^-$ we have $R \to \sqrt{\sigma^2 + 1}/\sigma \le \sqrt{5}/2 < 2/\sqrt{3}$ by the work above.  When $\sigma < 2$, we have $\rho = \sigma^2/4$, which occurs when $w_{14} = w_{23} = \sigma/2$.  Thus,
 \[ R^2 = \dfrac{\sigma^2 + 1}{(1 - \sigma^2/4)^2 + \sigma^2} = \dfrac{\sigma^2 + 1}{(1 + \sigma^2/4)^2}.\]
 As a function of $\sigma^2$, the maximum occurs when $\sigma^2 = 2$, and so $\sigma = \sqrt{2}$, $\rho = \sigma^2/4 = 1/2$, and $R = 2/\sqrt{3}$.  This corresponds to $w_{34} = 1$, $w_{14} = w_{23} = 1/\sqrt{2}$, and (by Proposition \ref{prop:dD12bound}) $w_{12} = 3$; indeed, $\dD(1,2) = \sqrt{3}$ and $d^{D'}(1,2) = 3/2$.
\end{proof}


\section{A sharp upper bound for $\sup \frac{\dD(x,y)}{d^{D'}(x,y)}$}
\label{sect:sharp}

As it turns out, the constant of $2/\sqrt{3}$ from Theorem \ref{thm:sharp} is the largest the ratio of $\dD(x,y)/d^{D'}(x,y)$ can be; this section is devoted to the proof of this result.  We begin with a technical lemma that will be useful.

\begin{lemma}
 \label{lem:ineq}
 Let $\delta, P, Q \in \R$ with $P,Q \ge 0$.  Then,
 \[ \sqrt{\delta^2 + P^2} + \sqrt{\delta^2 + Q^2} \le \frac{2}{\sqrt{3}}\sqrt{4\delta^2 + (P+Q)^2}.\]
\end{lemma}

\begin{proof}
 Since both sides of the inequality are symmetric in $P, Q$, we may assume $P \le Q$.  By the arithmetic-geometric mean inequality,
 \[ 2\sqrt{(\delta^2 + P^2)(\delta^2 + Q^2)} = 2\sqrt{3(\delta^2 + P^2) \cdot \frac{\delta^2 + Q^2}{3}} \le 3(\delta^2 + P^2) + \frac{\delta^2 + Q^2}{3},\]
 so
 \begin{align*}
  \left(\sqrt{\delta^2 + P^2} + \sqrt{\delta^2 + Q^2}\right)^2 &= (\delta^2 + P^2) + (\delta^2 + Q^2) +  2\sqrt{(\delta^2 + P^2)(\delta^2 + Q^2)}\\
  &\le 4(\delta^2 + P^2) + \frac{4}{3}(\delta^2 + Q^2)\\
  &= \frac{16}{3}\delta^2 + 4 P^2 + \frac{4}{3}Q^2.
 \end{align*}
On the other hand,
\[ \frac{4}{3}(4\delta^2 + (P+Q)^2) = \frac{16}{3}\delta^2 + \frac{4}{3}P^2 + \frac{8}{3}PQ + \frac{4}{3}Q^2,\]
and so
\[ \frac{4}{3}(4\delta^2 + (P+Q)^2) - (\sqrt{\delta^2 + P^2} + \sqrt{\delta^2 + Q^2})^2 \ge \frac{8}{3}P(Q - P) \ge 0.\]
Taking square roots proves the inequality.
\end{proof}

\begin{lemma}
 \label{lem:skewdeletionbound}
 Let $M = (m_{ij}) \in M_n(\R)$ be skew-symmetric, let $s \neq t$, and let $M'$ be obtained from $M$ by replacing the entries $m_{st}$ and $m_{ts}$ by $0$.  Then,
 \[ \Norm{M'} \le \frac{2}{\sqrt{3}} \Norm{M},\]
 and the constant $2/\sqrt{3}$ is sharp for every $n \ge 4$.
\end{lemma}

\begin{proof}
 We begin by proving that we can reduce to the $n = 4$ case with $s = 1$ and $t = 2$.  First, we may assume $n \ge 4$, since adjoining zero rows or columns changes neither $\Norm{M}$ nor $\Norm{M'}$.  Moreover, we can assume $M' \neq 0$.  Let $y,z \in \R^n$ be unit vectors with $y^tM'z = \Norm{M'}$.  If $e_s$ and $e_t$ denote standard basis vectors of $\R^n$, then the subspace spanned by $e_s, e_t, y, z$ has dimension at most $4$, so it is contained in a $4$-dimensional subspace $W$ with an orthonormal basis $f_1 = e_s$, $f_2 = e_t$, $f_3$, $f_4$.  Let $U$ be the $n \times 4$ matrix with columns $f_1, f_2, f_3, f_4$, and define $N \colonequals U^tMU \in M_4(\R)$.  Note that $N$ is skew-symmetric, since
 \[N^t = \left(U^tMU\right)^t = U^tM^tU = -U^tMU = -N.\]  Moreover, since $U$ has orthonormal columns,
 \[ \Norm{N} \le \Norm{U^t}\Norm{M}\Norm{U} = \Norm{M},\] and its $1,2$-entry is $e_s^tMe_t = m_{st}$.  Thus,
 \[ U^tM'U = U^tMU - m_{st}U^t(e_se_t^t - e_te_s^t)U = N',\]
 where $N'$ is obtained from $N$ by replacing the $1,2$- and $2,1$-entries by $0$.  There exist unit vectors $\hat{y}, \hat{z} \in \R^n$ such that $y = U\hat{y}$, $z = U\hat{z}$, so
 \[ \Norm{M'} = y^t M' z = \hat{y}^tN'\hat{z} \le \Norm{N'}.\]  Hence, it suffices to prove $\Norm{N'} \le \frac{2}{\sqrt{3}}\Norm{N}$ for skew-symmetric $N \in M_4(\R)$ with $s = 1$ and $t = 2$.

 Let $N = (n_{ij})$.  Using the notation from Section \ref{sect:prelim}, we have
 \[F = F(N) = n_{12}^2 + n_{13}^2 + n_{14}^2 + n_{23}^2 + n_{24}^2 + n_{34}^2\]
 and
 \[ \Pf = \Pf(N) = n_{12}n_{34} - n_{13}n_{24} + n_{14}n_{23}.\]
 Defining
 \[ \eta \colonequals \frac{1}{2}(n_{12} + n_{34}, n_{13} - n_{24}, n_{14} + n_{23}), \quad \zeta \colonequals \frac{1}{2}(n_{12} - n_{34}, n_{13} + n_{24}, n_{14} - n_{23}) \in \R^3,\]
 we see that $|\eta|^2 + |\zeta|^2 = F/2$ and $|\eta|^2 - |\zeta|^2 = \Pf$, so
 \[ 4|\eta|^2|\zeta|^2 = \left(\frac{F}{2} \right)^2 - \Pf^2\]
 and
 \[ \left(|\eta| + |\zeta| \right)^2 = \frac{F}{2} + \sqrt{\frac{F^2}{4} - \Pf^2} = \frac{F + \sqrt{F^2 - 4\Pf^2}}{2} = s_1^2,\]
 the largest root of $x^2 - Fx + \Pf^2$ (see Proposition \ref{prop:skew}).  Hence,
 \begin{equation}
  \label{eq:normN}
  \Norm{N} = |\eta| + |\zeta|.
 \end{equation}
 Deleting $n_{12}$ replaces $\eta$ by $\eta' = \eta - (n_{12}/2,0,0)$ and $\zeta$ by $\zeta' = \zeta - (n_{12}/2, 0, 0)$; the first coordinates become, respectively, $n_{34}/2$ and $-n_{34}/2$, and the other coordinates do not involve $n_{12}$.  Define
 \[\delta \colonequals \frac{n_{34}}{2}, \quad P \colonequals \Norm{\frac{1}{2}(n_{13} - n_{24}, n_{14} + n_{23})}, \quad Q \colonequals \Norm{\frac{1}{2}(n_{13} + n_{24}, n_{14} - n_{23})}.\]
 Applying \eqref{eq:normN} to $N'$ and $N$, we have
 \begin{align*} \Norm{N'} &= \sqrt{\delta^2 + P^2} + \sqrt{\delta^2 + Q^2},\\
  \Norm{N} &= |(\frac{n_{12}}{2} + \delta, \frac{n_{13} - n_{24}}{2}, \frac{n_{14} + n_{23}}{2})| + |(\frac{n_{12}}{2} - \delta, \frac{n_{13} + n_{24}}{2}, \frac{n_{14} - n_{23}}{2})|\\
  &\ge |(2\delta, P + Q)| = \sqrt{4\delta^2 + (P+Q)^2};\\
 \end{align*}
 that
 \[\Norm{N'} \le \frac{2}{\sqrt{3}}\Norm{N}\]
 now follows from Lemma \ref{lem:ineq}.

 Finally, to get the sharpness of the inequality, take $n_{12} = 1/\sqrt{3}$, $n_{13} = 0$, $n_{14} = \sqrt{2/3}$, $n_{23} = -\sqrt{2/3}$, $n_{24} = 0$, and $n_{34} = -1/\sqrt{3}$.  Then, $\eta = (0,0,0)$ and $\zeta = (1/\sqrt{3}, 0, \sqrt{2/3})$, so $\Norm{N} = 1$ by \eqref{eq:normN}, while $\eta' = (-\frac{1}{2\sqrt{3}}, 0, 0)$ and $\zeta' = (\frac{1}{2\sqrt{3}}, 0, \sqrt{2/3})$, so \[\Norm{N'} = \frac{1}{2\sqrt{3}} + \frac{\sqrt{3}}{2} = \frac{2}{\sqrt{3}}.\]
 We adjoin zero rows and columns when the matrix has dimension larger than $4$.
\end{proof}

\begin{proof}[Proof of Theorem \ref{thm:supbound}]
 First, if $x$ and $y$ lie in different components of the graph after deletion, then $d^{D'}(x,y) = +\infty$, and there is nothing to prove.  Otherwise, let $A$ be admissible for $D$.  Since $[A,D']_{jk} = (a_j - a_k)D'_{jk}$, the matrix $[A,D']$ is obtained from the skew-symmetric matrix $[A,D]$ by replacing its two entries on $e$ by $0$, so, by Lemma \ref{lem:skewdeletionbound}, $\Norm{[A,D']} \le 2/\sqrt{3}$.  Hence, $\frac{\sqrt{3}}{2}A$ is always admissible for $D'$, and $d^{D'}(x,y) \ge \frac{\sqrt{3}}{2}|a_x - a_y|$.  Taking the supremum over $A$ proves the inequality.  Finally, the weighted $4$-cycle of Theorem \ref{thm:sharp} has $\dD(1,2) = \sqrt{3}$ and $d^{D'}(1,2) = 3/2$, showing that the bound is sharp.
\end{proof}


\section{Edge deletion in cycles}
\label{sect:cycles}

The purpose of this section is to study edge deletion in cycles of length longer than $4$.  Necessarily, the situation is more difficult: precise formulas for the noncommutative distance are not known for arbitrary paths, let alone arbitrary cycles; see \cite{CLPS25}.  We nonetheless can still determine for precisely which values of $n$ equation \eqref{eq:question} fails to hold for $n$-cycles.  Surprisingly, \eqref{eq:question} only fails for $n$-cycles when $n$ is divisible by $4$.

\subsection{Results on general skew-symmetric matrices}

We first prove some results that hold for general skew-symmetric matrices.  Let $M \in M_n(\R)$ be skew-symmetric and let $r = \lfloor n/2 \rfloor$.  Recall that the eigenvalues of $M$ are $\pm i s_1$, \dots, $\pm i s_r$, along with $0$ when $n$ is odd, where $s_1 \ge \cdots \ge s_r \ge 0$ are the singular values of $M$.  This implies $\Norm{M} = s_1$.

For a subset $S \subseteq \{1, \dots, n\}$, we denote by $M_S$ the principal submatrix of $M$ with rows and columns indexed by $S$; note that $M_S$ itself is skew-symmetric.  For $0 \le k \le r$, if $e_k$ denotes the $k^{\text{th}}$ elementary symmetric polynomial, we define
\[ \ep_k(M) \colonequals e_k(s_1^2, \dots, s_r^2).\]

The following result is classical and follows from the relation between principal minors and elementary symmetric polynomials (see \cite[Section 1.2]{HJ13}) and the relation $\det A = \Pf(A)^2$.

\begin{lemma}
 \label{lem:pfaffsympoly}
 For every skew-symmetric $M \in M_n(\R)$ and every $0 \le k \le r$, we have
 \[  \ep_k(M) = \sum_{\substack{S \subseteq \{1, \dots, n\},\\ |S| = 2k}} \Pf(M_S)^2.\]
\end{lemma}

For $0 \le k \le r$, define
\[ \Xi_k(M) \colonequals e_k(1-s_1^2, \dots, 1 - s_r^2);\]
in particular, $\Xi_0(M) = 1$.

\begin{lemma}
 \label{lem:normMandXi}
 For every skew-symmetric $M \in M_n(\R)$ and $1 \le k \le r$,
 \[ \Xi_k(M) = \sum_{j = 0}^k (-1)^j {{r - j} \choose {k - j}} \ep_j(M),\]
 and $\Norm{M} \le 1$ if and only if $\Xi_k(M) \ge 0$ for $1 \le k \le r$.
\end{lemma}

\begin{proof}
 First, we note that the generating function for the elementary symmetric functions in variables $z_1, z_2, \dots, z_r$ is
 \[ P(t) = \sum_{k = 0}^r e_k(z_1, \dots, z_r) t^k = \prod_{j = 1}^r (1 + z_jt);\]
 see \cite[p. 20]{Mac15}.
 After substituting $z_j = 1 - x_j$, we have
 \[ P(t) = (1 + t)^r \prod_{j=1}^r \left(1 - \frac{t}{1 + t}x_j \right) = \sum_{j = 0}^r (-1)^j e_j(x_1, \dots, x_r)t^j(1+t)^{r-j}.\]
 Expanding $(1+t)^{r-j}$ using the Binomial Theorem, we get
 \[P(t) = \sum_{j = 0}^r (-1)^j e_j(x_1,\dots, x_r)t^j \sum_{m = 0}^{r-j} {{r-j} \choose m} t^m, \]
 and setting $j + m = k$, after comparing coefficients, we get
 \[ e_k(1- x_1, \dots, 1 - x_r) = \sum_{j = 0}^k (-1)^j {{r-j} \choose {k-j}} e_j(x_1, \dots, x_r).\]
 Setting $x_j = s_j^2$ for each $j$ so that $e_j(x_1, \dots, x_r) = \ep_j(M)$, we see that
 \[ \Xi_k(M) = \sum_{j = 0}^k (-1)^j {{r - j} \choose {k - j}} \ep_j(M).\]

 Assume $\Norm{M} \le 1$, i.e., $s_j \le 1$ for each $j$.  Since $1 - s_j^2 \ge 0$ for each $j$, $\Xi_k(M) = e_k(1 - s_1^2, \dots, 1 - s_r^2)$ is a sum of products of nonnegative numbers and hence is itself nonnegative.  Conversely, assume $\Xi_k(M) \ge 0$ for all $k$.  Define
 \[ p(t) = \prod_{j = 1}^r (t - (1 - s_j^2)) = \sum_{k=0}^r (-1)^k \Xi_k(M) t^{r - k},\]
 a monic polynomial with real roots, and assume $t < 0$.  Then, for each $k$, $(-1)^k \Xi_k(M) t^{r - k}$ has sign $(-1)^r$ or vanishes.  Since $p(t)$ is monic, $p$ is not identically $0$, and so $p(t)$ has sign $(-1)^r$, i.e., $p(t) \neq 0$ for all $t < 0$.  Hence, $p(t)$ has no negative roots, i.e., $1 - s_j^2 \ge 0$ for all $j$, and so $\Norm{M} = s_1 \le 1$, as desired.
\end{proof}

The following result is a weighted form of a known result; it follows from the analysis of the weighted matching polynomial studied in \cite{GX12} (unweighted versions were studied independently in \cite{Cav12, HL11}).  We include a full proof for the sake of completeness.

\begin{lemma}
 \label{lem:matchings}
 Let $M$ be supported on a graph $G = (V,E)$ (that is, $m_{ij} = 0$ whenever $\{i,j\} \notin E$), and suppose that no subset of size $2k$ of $V$ induces a subgraph of $G$ with two distinct perfect matchings.  Then
 \[ \ep_k(M) = \sum_{F} \prod_{f \in F} m_f^2,\]
 where the sum is over the sets $F$ of $k$ pairwise disjoint edges of $G$ and $m_f$ denotes the entry on the edge $f$.  In particular, the hypothesis holds when $G$ is a disjoint union of paths and when $G$ is an odd cycle.
\end{lemma}

\begin{proof}
 Recall the definition of the Pfaffian \eqref{eq:Pf}.  Now, fix $k$ and let $S \subseteq V$ with $|S| = 2k$.  Since $M$ is supported on $G$, the terms of \eqref{eq:Pf} for $N = M_S$ vanish unless every edge of $F$ is an edge of $G$, i.e., unless $F$ is a perfect matching of the induced subgraph $G[S]$.  If we denote by $\PM(H)$ the set of all perfect matchings of a graph $H$,  \eqref{eq:Pf} becomes
\[ \Pf(M_S) = \sum_{F \in \PM(G[S])} \sgn(F) \prod_{f \in F} m_f,\]
which equals $0$ when $G[S]$ has no perfect matching.  By hypothesis, no subset of $V$ induces more than one perfect matching, so, after squaring, we have
\begin{equation}
 \label{eq:Pf2}
 \Pf(M_S)^2 = \begin{cases}
               \prod_{f \in F} m_f^2, &\PM(G[S]) = \{F\},\\
               0, &\text{ otherwise.}
              \end{cases}
\end{equation}
Combining \eqref{eq:Pf2} with Lemma \ref{lem:pfaffsympoly}, we have
\[ \ep_k(M) = \sum_{\substack{S \subseteq \{1, \dots, n\},\\ |S| = 2k}} \Pf(M_S)^2 = \sum_{F} \prod_{f \in F} m_f^2,\]
 where the last sum is over the sets $F$ of $k$ pairwise disjoint edges of $G$.
\end{proof}

\subsection{Edge deletion in odd cycles.}

We now have enough to prove that equation \eqref{eq:question} holds for edge deletion in odd cycles.

\begin{theorem}
 \label{thm:odd}
 Let $n$ be odd, let $D$ be a Dirac operator on the $n$-cycle, let $e$ be any edge, and let $D'$ be obtained from $D$ by deleting $e$.  Every $A$ admissible for $D$ is also admissible for $D'$, and for any vertices $x$ and $y$,
 \[ d^{D'}(x,y) \ge d^D(x,y).\]
\end{theorem}

\begin{proof}
 Let $n = 2r + 1$ and label the cycle $1 - 2 - \cdots - n - 1$, with $e = \{1, 2\}$.  Let $\Norm{M} \le 1$, where $M = [A,D]$. Thus, $M' = [A,D']$ is $M$ with $m_{12}$ replaced by $0$, and it suffices to prove $\Norm{M'} \le 1$.

 Both the odd cycle and the path remaining after the deletion of $e$ satisfy the hypothesis of Lemma \ref{lem:matchings}, and a set of $k$ disjoint edges of the cycle either avoids $e$ or consists of $e$ together with $k-1$ edges disjoint from $e$.  Let $S = \{3, 4, \dots, n\}$; note that $M_S$ is the matrix of the path $3 - 4 - \cdots - n$, whose edges are precisely those edges of the cycle that do not share a vertex with $e$.  Thus, by Lemma \ref{lem:matchings}, for $1 \le j \le r$,
 \[ \ep_j(M) = \ep_j(M') + m_{12}^2\ep_{j-1}(M_S),\]
 i.e., $\ep_j(M') = \ep_j(M) - m_{12}^2 \ep_{j-1}(M_S)$.  Thus, for $1 \le k \le r$,
 \begin{align*} \Xi_k(M') - \Xi_k(M) &= - \sum_{j = 1}^k (-1)^j {{r-j} \choose {k-j}} m_{12}^2 \ep_{j-1}(M_S)\\
 &= m_{12}^2 \sum_{i = 0}^{k-1}(-1)^i {{(r-1) -i} \choose {(k-1) - i}}\ep_i(M_S)\\
 &= m_{12}^2 \Xi_{k-1}(M_S) .
 \end{align*}
 By Lemma \ref{lem:compression}, $\Norm{M_S} \le \Norm{M} \le 1$, so  by Lemma \ref{lem:normMandXi} $\Xi_{k-1}(M_S) \ge 0$ for $1 \le k-1 \le r-1$, and $\Xi_0(M_S) = 1 > 0$.  Thus, $\Xi_k(M') \ge \Xi_k(M) \ge 0$ for every $k$, and $\Norm{M'} \le 1$ by Lemma \ref{lem:normMandXi}.
\end{proof}

\begin{example}
 \label{ex:C5}
 For $n = 5$, we have $r = 2$, $\ep_0(M) = 1$, \[\ep_1(M) = m_{12}^2 + m_{23}^2 + m_{34}^2 + m_{45}^2 + m_{51}^2,\] \[\ep_2(M) = m_{12}^2 m_{34}^2 + m_{23}^2 m_{45}^2 + m_{34}^2 m_{51}^2 + m_{45}^2 m_{12}^2 + m_{51}^2m_{23}^2,\]
 \[\Xi_1(M) = {2 \choose 1}\ep_0(M) - {1 \choose 0}\ep_1(M) = 2 - \ep_1(M),\]
 \[ \Xi_2(M) = {2 \choose 2}\ep_0(M) - {1 \choose 1}\ep_1(M) + {0 \choose 0}\ep_2(M) = 1 - \ep_1(M) + \ep_2(M).\]

 Set $S = \{3,4,5\}$.  Deleting $e = \{1, 2\}$ removes the term $m_{12}^2$ from $\ep_1(M)$ and the terms $m_{12}^2 m_{34}^2$ and $m_{45}^2 m_{12}^2$ from $\ep_2(M)$. Since
 \begin{align*} \Xi_1(M') &= 2 - \ep_1(M')\\
                          &= 2 - \ep_1(M) + m_{12}^2\\
                          &= \Xi_1(M) + m_{12}^2\Xi_0(M_S)\\
                          &\ge \Xi_1(M)\\
                          &\ge 0,
 \end{align*}
 and
 \begin{align*} \Xi_2(M') &= 1 - \ep_1(M') + \ep_2(M')\\
                          &= 1 - \ep_1(M) + \ep_2(M) + m_{12}^2 - m_{12}^2 m_{34}^2 - m_{45}^2 m_{12}^2\\
                          &= \Xi_2(M) + m_{12}^2\Xi_1(M_S)\\
                          &\ge \Xi_2(M)\\
                          &\ge 0,
 \end{align*}
by Lemma \ref{lem:normMandXi}, $\Norm{M'} \le 1$, and any $A$ admissible for a Dirac operator $D$ of the $5$-cycle is admissible for the Dirac operator $D'$ obtained from $D$ by deleting $e$.
\end{example}

\subsection{Edge deletion in $n$-cycles, $n \equiv 2 \pmod 4$.}

The situation for even cycles is more complicated.  When $n = 2r$ and $k < r$, the hypotheses of Lemma \ref{lem:matchings} still hold, and so in this case (using the notation of the proof of Theorem \ref{thm:odd}, where $S = \{3, 4, \dots, n\}$),
\[ \Xi_k(M') = \Xi_k(M) + m_{12}^2 \Xi_{k-1}(M_S) \ge \Xi_k(M) \ge 0.\]
Thus, for even cycles, it suffices to consider the case when $k = r$; however, in this case, the graph has two perfect matchings, and Lemma \ref{lem:matchings} no longer applies.  In this case, we have
\[ \ep_r(M) = \Pf(M)^2 = (m_{12}m_{34} \cdots m_{2r-1, 2r} - m_{23}m_{45} \cdots m_{2r, 1})^2,\]
and, after deleting the edge $e = \{1,2\}$, we have
\begin{align*}
\ep_r(M') - \ep_r(M) &= (m_{23}m_{45} \cdots m_{2r, 1})^2 - \ep_r(M)\\
      &= -m_{12}^2 \cdots m_{2r-1,2r}^2 + 2m_{12}m_{23}m_{34} \cdots m_{2r, 1}.
\end{align*}
Thus, using Lemma \ref{lem:normMandXi} we have
\begin{equation}
\label{eq:Xir}
\Xi_r(M') - \Xi_r(M) = m_{12}^2\Xi_{r-1}(M_S) + (-1)^r 2m_{12}m{23} \cdots m_{2r,1}.
\end{equation}

Indeed, this means it is possible that certain matrices $A$ admissible for a Dirac operator for an $n$-cycle may no longer be admissible after edge deletion.  We first prove a sufficient condition for such a matrix $A$ to remain admissible after edge deletion, which essentially follows from \eqref{eq:Xir}.

\begin{proposition}
 \label{prop:signcrit}
 Let $n = 2r$, let $D$ be a Dirac operator on the $n$-cycle, let $A$ be admissible for $D$, let $M = [A, D]$, and let $M'$ be obtained by deleting an edge.  If
 \[ (-1)^r m_{12}m_{23} \cdots m_{2r, 1} \ge 0,\]
 then $\Norm{M'} \le 1$; that is, $A$ remains admissible after the edge deletion.
\end{proposition}

\begin{proof}
  Relabeling the cycle if necessary, we may assume $e = \{1,2\}$; note that this does not change the product $m_{12}m_{23} \cdots m_{2r, 1}$.  As noted above, when $k < r$,
  $\Xi_k(M') \ge 0$, following the proof of Theorem \ref{thm:odd} mutatis mutandis.  When $k = r$, by \eqref{eq:Xir}, we have
  \[ \Xi_r(M') - \Xi_r(M) = m_{12}^2\Xi_{r-1}(M_S) + (-1)^r 2m_{12}m_{23} \cdots m_{2r,1},\]
  and both terms on the righthand side are nonnegative under the hypothesis.  By Lemma \ref{lem:normMandXi}, $\Norm{M'} \le 1$.
\end{proof}

For $x \in \R^{n-1}$ define $L(x)$ to be the skew-symmetric matrix of the path $2 - 3 - \cdots - n - 1$ with entries from $x$: namely, $L(x)_{i,i+1} = x_i$ for the edge $\{i,i+1\}$, and $L(x)_{i,j} = 0$ when $i \le j$ otherwise.

\begin{proposition}
 \label{prop:reduction}
 Let $n = 2r,$ let $D$ be any Dirac operator on the $n$-cycle, let $e = \{1,2\}$, let $A$ be admissible for $D$ with $a_1 - a_2 = \dD(1,2)$, let $M = (m_{i,j}) = [A,D]$, and let $\mu_i = \max(-m_{i,i+1},0)$ for $2 \le i \le n$, where the subscripts are read modulo $n$, and set $\mu = (\mu_2, \dots, \mu_n)$.  Then, $m_{12} \ge 0$ and
 \[ \dD(1,2) \le \Norm{L(\mu)} d^{D'}(1,2).\]
\end{proposition}

\begin{proof}
 For $1 \le i \le n$, we set $z_i = a_i - a_{i+1}$; here, the indices are read modulo $n$, so $z_n = a_n - a_1$.  Thus, \[m_{i, i+1} = (a_i - a_{i+1})d_{i,i+1} = \frac{z_i}{w_{i,i+1}}.\]
 Moreover, since
 \[ \sum_{i = 1}^n z_i = 0\]
 and $\dD(1,2) = z_1 \ge 0$, we have
 \begin{equation}
 \label{eq:z1}
 \dD(1,2) = z_1 = \sum_{i = 2}^n w_{i,i+1}(-m_{i,i+1}) \le \sum_{i = 2}^n w_{i,i+1} \mu_i.
  \end{equation}
Furthermore, $m_{12} \ge 0$ since $z_1, w_{12} \ge 0$.

 Assume first that $L(\mu) = 0$.  Since the nonzero entries of $L(\mu)$ are $\pm \mu_i$, this means $\mu = 0$, and so $m_{i,i+1} \ge 0$ for $i \ge 2$.  Hence $z_1 \le 0$, i.e., $\dD(1,2) = z_1 = 0$.  Since $0 \le 0 \cdot d^{D'}(1,2)$, the result holds trivially in this case.

 Otherwise, set $x = \mu/\Norm{L(\mu)}$, so $x_i \ge 0$ for each $i$ and $\Norm{L(x)} = 1$.  Deleting $e$ leaves the path $2 - 3 - \cdots - n - 1$, whose edges are $e_{23}, \dots, e_{n,1}$.  Define $B \colonequals \diag(b_1, b_2, \dots, b_n)$, where $b_2 = 0$ and $b_{i+1} = b_i - w_{i,i+1}x_i$ for $2 \le i \le n$ (again, subscripts are read modulo $n$).  Note that the entry of $[B,D']$ on $e_{i,i+1}$ for $2 \le i \le n$ is
 \[ (b_i - b_{i+1})d_i = w_{i,i+1}x_i/w_{i,i+1} = x_i.\]  Since $[B,D']$ has no other nonzero entries, $[B,D'] = L(x)$, i.e., $\Norm{[B,D']} = \Norm{L(x)} = 1$, and $B$ is admissible for $D'$.  Telescoping along the path from vertex $2$ to vertex $1$, we have
 \[ d^{D'}(1,2) \ge |b_1 - b_2| = \left|\sum_{i=2}^{n} (b_i - b_{i+1})\right| = \sum_{i = 2}^n w_{i,i+1} x_i = \frac{1}{\Norm{L(\mu)}} \sum_{i = 2}^n w_{i,i+1} \mu_i \ge  \frac{\dD(1,2)}{\Norm{L(\mu)}},\]
 completing the proof.
\end{proof}

We can now prove that \eqref{eq:question} holds for edge deletion in $2r$-cycles, where $r$ is odd.

\begin{theorem}
 \label{thm:2mod4}
 Let $n = 2r$, where $r$ is odd.  Let $D$ be a Dirac operator on an $n$-cycle, $e$ be any edge, and $D'$ the Dirac operator obtained by deleting $e$.  For any vertices $x$ and $y$,
 \[ d^{D'}(x,y) \ge \dD(x,y). \]
\end{theorem}

\begin{proof}
 Let $n = 2r$ and label the cycle $1 - 2 - \cdots - n - 1$, and without loss of generality let $e = \{1, 2\}$.  We first treat the pair $\{x,y\} = \{1,2\}$.  Let $A$ be admissible for $D$ with $a_1 - a_2 = \dD(1,2)$; such an $A$ exists because the supremum in \eqref{eq:def} is attained (after normalizing $a_2 = 0$, the constraints $|a_i - a_{i+1}| \le w_{i,i+1}$ confine the admissible diagonals to a compact set, on which $a_1 - a_2$ is continuous). Let $M = (m_{i,j}) = [A,D]$.  By Proposition \ref{prop:reduction}, we have $m_{12} \ge 0$.  Let $\mu_i = \max(-m_{i,i+1}, 0)$ for $2 \le i \le n$ (again, $m_{n, n+1} = m_{n,1}$) and set $\mu = (\mu_2, \dots, \mu_n)$.

 Suppose first that $m_{j,j+1} \ge 0$ for some $2 \le j \le n$.  Then, $\mu_j = 0$, and $L(\mu)$ is the direct sum of blocks corresponding to the maximal runs of consecutive nonzero entries of $\mu$, together with a zero block: the vertex sets of distinct runs are disjoint and no nonzero entry of $L(\mu)$ joins two of them, so the coordinate subspace spanned by the vertices of each run is invariant, and $\Norm{L(\mu)}$is the largest of the norms of the blocks.  It thus suffices to prove that $\Norm{L(\nu)} \le 1$ for each run $\nu = (\mu_i, \mu_{i+1}, \dots, \mu_{i + k - 1}) \in \R^k$, all of whose entries are nonzero, where $k \le n - 1$ and $L(\nu)$ denotes the skew-symmetric matrix of the path $i - (i + 1) - \cdots - (i + k)$ with entries $\nu$ and indices read modulo $n$.  In particular, since $\mu_j = 0$, it is impossible that both $i = 2$ and $i + k - 1 = n$.  Thus, if $T = \{i , i + 1, \dots, i + k\}$ is the set of the $k+1$ vertices of the run, then $T$ does not contain both $1$ and $2$.  Hence, the only edges of the cycle with both vertices in $T$ are $\{i, i+1\}, \dots, \{i + k -1, i + k\}.$  Since $\mu_\ell = -m_{\ell, \ell+1}$ for each $\ell$ in the run, $L(\nu) = -M_T$, and by Lemma \ref{lem:compression},
 \[ \Norm{L(\nu)} = \Norm{M_T} \le \Norm{M} \le 1.\]
 Hence, $\Norm{L(\mu)} \le 1$ in this case, so by Proposition \ref{prop:reduction},
 \[ \dD(1,2) \le \Norm{L(\mu)} d^{D'}(1,2) \le d^{D'}(1,2).\]

 Now suppose $m_{j,j+1} < 0$ for $2 \le j \le n$.  We already know that $m_{12} \ge 0$, and $r$ is odd, so
 \[ (-1)^rm_{12}m_{23} \cdots m_{n,1} \ge 0. \]
 By Proposition \ref{prop:signcrit}, $\Norm{M'} \le 1$, meaning $A$ is admissible for $D'$ as well, and hence
 \[ \dD(1,2) \le d^{D'}(1,2)\]
 in any case.

 Finally, let $\{x,y\} \neq \{1,2\}$.  The graph obtained by deleting $e$ is the path $2 - 3 - \cdots - n - 1$, and the union of the simple paths from $x$ to $y$ is the path between $x$ and $y$, which contains $1$ and $2$ only when $\{x,y\} = \{1,2\}$.  The path is therefore an induced subgraph of the cycle, so, by Propositions \ref{prop:deletion} and \ref{prop:prune}, $d^{D'}(x,y) \ge \dD(x,y)$, completing the proof.
\end{proof}

\subsection{Edge deletion in $n$-cycles with $n \equiv 0 \pmod 4$.}

We can see in the proof of Theorem \ref{thm:2mod4} where things can break down when $n$ is divisible by $4$: if $m_{j,j+1} < 0$ for $2 \le j \le n$, we have
\[ (-1)^rm_{12}m_{23} \cdots m_{n,1} < 0 \]
if $m_{12} \neq 0$, since $m_{12} \ge 0$ and $r$ is even.  Indeed, it is possible to construct an example where \eqref{eq:question} fails when $n$ is divisible by $4$, and the purpose of this subsection is to provide such an example for each such $n$.

Throughout this subsection, we let $n = 2r$, where $r$ is even.  Let
\[ \theta = \frac{\pi}{n+1}, \quad c^* = \frac{1}{2\cos \theta}.\]
The $n$-cycle is labeled $1 - 2 - \cdots - n - 1$, where the edge $\{i, i+1\}$ has weight $w_{i,i+1}$ with $d_{i,i+1} = w_{i,i+1}^{-1}$, and the deleted edge is $e = \{1,2\}.$

\begin{lemma}
 \label{lem:weighting}
 Let $M = (m_{ij})$ be a skew-symmetric matrix such that $m_{ij} = 0$ unless $j \equiv i \pm 1 \pmod n$, and let $w_{12}, \dots, w_{n,1}$ be a choice of weights for the $n$-cycle, where $w_{j,j+1}$ is the weight of $\{j, j+1\}$.  If $\Norm{M} \le 1$ and $\sum_{i = 1}^n w_{i,i+1} m_{i,i+1} = 0$, then there exists an admissible diagonal matrix $A$ for $D$ such that $[A,D] = M$, and $\dD(1,2) \ge w_{12}m_{12}$.
\end{lemma}

\begin{proof}
 We construct a specific admissible $A$ giving us $[A,D] = M$.  Define $a_1 \colonequals 0$ and $a_{i+1} \colonequals a_i - w_{i,i+1} m_{i,i+1}$ for $1 \le i \le n-1$, and let $A = \diag(a_1, \dots, a_n)$; indeed,
 \[ m_{i,i+1} = w_{i,i+1}m_{i,i+1}d_{i,i+1} = (a_i - a_{i+1})d_{i,i+1},\]
 and, since $\sum_{i = 1}^n w_{i,i+1} m_{i,i+1} = 0$,
 \[ m_{n,1} = d_{n,1}(w_{n,1} m_{n,1}) = d_{n,1}\left(-\sum_{i = 1}^{n-1}w_{i, i+1}m_{i, i+1}\right) = d_{n,1}a_{n} = (a_n - a_1)d_{n,1}.\]  All other entries of $[A,D]$ vanish, and $[A,D]$ is skew-symmetric; hence $[A,D] = M$, $A$ is admissible for $D$, and $\dD(1,2) \ge a_1 - a_2 = w_{12}m_{12}$.
\end{proof}

In the next two lemmas, the rows and columns of $L(x)$ are indexed in the order of the path as in the proof of Lemma \ref{lem:weighting}, i.e.,
\[ u_1 = 2, \quad u_2 = 3, \quad \dots, \quad u_{n-1} = n, \quad u_n = 1;\]
note that simultaneously permuting the rows and columns in this way does not change norms, singular values, or expressions of the form $y^t L(x) z$ when $y$ and $z$ are indexed in the same way as $L(x)$.  With respect to this order, $L(x)$ is tridiagonal with $L(x)_{u_j, u_{j+1}} = x_{j+1} = -L(x)_{u_{j+1}, u_j}$ for $1 \le j \le n-1$.  Recalling that $\theta = \pi/(n+1)$, we write $h_j = \sin(j\theta)$ for $0 \le j \le n+1$, so that $h_0 = h_{n+1} = 0$, $h_j > 0$ for $1 \le j \le n$, and $h_{j-1} + h_{j+1} = 2\cos(\theta)h_j$ for $1 \le j \le n$.

\begin{lemma}
 \label{lem:constantpath}
 Let $c > 0$ and $\mathbf{1} \in \R^{n-1}$ be the all-ones vector.  For $1 \le k \le n$ the vector $g^{(k)} \in \C^n$ with entries $g_j^{(k)} = \ii^{j-1}\sin(jk\theta)$ satisfies
 \[ L(c\mathbf{1})g^{(k)} = 2\ii c \cos(k\theta) g^{(k)}.\]
 Consequently, the singular values of $L(c\mathbf{1})$ are the numbers $s_k \colonequals 2c\cos(k\theta)$, $1 \le k \le r$, each with multiplicity two, and they are positive and pairwise distinct.  Thus, $\Norm{L(c\mathbf{1})} = 2c\cos \theta$, and, if $c^* = \frac{1}{2\cos(\theta)}$, $\Norm{L(c^*\mathbf{1})} = 1$ and, in this case, for $2 \le k \le r$, the singular values are
 \[ s_k = \frac{\cos(k\theta)}{\cos \theta} < 1.\]
\end{lemma}

\begin{proof}
 This is the idea used in the proof of \cite[Lemma 3.3]{CLPS25}.  If we define
 \[ Q \colonequals \diag(\ii, \ii^2, \dots, \ii^n),\]
 indexed along the path, then the matrix $\ii Q^* L(c\mathbf{1}) Q$ is a real symmetric tridiagonal Toeplitz matrix, with all off-diagonal entries equal to $-c$.  We write $\omega_j = \sin(jk\theta)$, so that $\omega_0 = \omega_{n+1} = 0$ and $\omega_{j-1} + \omega_{j+1} = 2\cos(k\theta)\omega_j$.  With the convention $g_0^{(k)} = g_{n+1}^{(k)} = 0$, which is consistent with $\omega_0 = \omega_{n+1} = 0$, for $1 \le j \le n$, we have
 \begin{align*}
  \left(L(c\mathbf{1})g^{(k)}\right)_j &= cg_{j+1}^{(k)} - cg_{j-1}^{(k)} = c\left(\ii^j \omega_{j+1} - \ii^{j-2}\omega_{j-1} \right)\\
   &= c \ii^{j-1}\ii\left(\omega_{j+1} + \omega_{j-1} \right) = 2\ii c \cos(k\theta)g_j^{(k)},\\
 \end{align*}
using $-\ii^{j-2} = \ii^{j-1} \cdot \ii$.  For $1 \le k \le n$, the values of $\cos(k \theta)$ are distinct since $0 < k\theta < \pi$, so this exhibits $n$ eigenvectors for $n$ distinct eigenvalues; hence, the spectrum of $L(c\mathbf{1})$ is $\{2 \ii c \cos(k\theta): 1 \le k \le n\}$.  Since $\cos((n+1 - k)\theta) = -\cos(k\theta)$, the spectrum can be expressed as $\{\pm 2\ii c \cos(k\theta): 1 \le k \le r\}$, and $\cos(k\theta) > 0$ for $k \le r$ because $k\theta < \pi/2$ exactly when $2k < n+1$.  The result follows.
\end{proof}

Note that $0 < (i-1)\theta < i\theta \le n\theta < \pi$ when $2 \le i \le n$, so $\sin((i-1)\theta)\sin(i\theta) > 0$.

\begin{lemma}
 \label{lem:sine}
 For $1 \le i \le n$, define
 \[ \varpi_i \colonequals \sin((i-1)\theta)\sin(i\theta).\]
 If the path $2 - 3- \cdots - n - 1$ has a weight of $\varpi_j$ on the edge $\{j, j+1\}$ (where $\{n, n+1\} = \{n,1\}$), then
 \[ d^{D'}(1,2) = c^* \sum_{i = 1}^n \varpi_i.\]
\end{lemma}

\begin{proof}
 We proceed using the results of \cite[Section 3]{CLPS25} applied to the path with vertices $u_1 = 2, u_2 = 3, \dots, u_n = 1$; its $j^{\text{th}}$ edge is $\{j+1, j+2\}$ (read modulo $n$), and we write $d_j' = \varpi_{j+1}^{-1}$ for its Dirac entry, so that the vector $d$ as in \cite[Definition 3.4]{CLPS25} is $(d_1', \cdots, d_{n-1}')$.  Since
 \[ \frac{d_j'}{d_{j+1}'} = \frac{\varpi_{j+2}}{\varpi_{j+1}} = \frac{\sin((j+1)\theta)\sin((j+2)\theta)}{\sin(j\theta)\sin((j+1)\theta)} = \frac{h_{j+2}}{h_j},\]
 the products defining the vectors $\mu, \nu \in \R^r$ of \cite[Definition 3.4]{CLPS25} telescope:
 \[ \mu_{k+1} = \prod_{j=1}^k \frac{d_{2j-1}'}{d_{2j}'} = \prod_{j=1}^k \frac{h_{2j+1}}{h_{2j-1}} = \frac{h_{2k+1}}{h_1}, \quad \nu_{k+1} = \prod_{j=1}^k \frac{d_{2j}'}{d_{2j+1}'} = \frac{h_{2k+2}}{h_2}.\]
 Thus $\mu$ and $\nu$ are proportional to $h_{\text{odd}} = (h_1, h_3, \dots, h_{n-1})$ and $h_{\text{even}} = (h_2, h_4, \dots, h_n)$, and the unit vectors of \cite[Definition 3.18]{CLPS25} are $x = h_{\text{odd}}/\Norm{h_{\text{odd}}}$ and $y = h_{\text{even}}/\Norm{h_{\text{even}}}$.  Moreover, since $n+1$ is odd, the map $j \mapsto n + 1 - j$ exchanges odd and even indices, so $h_{n+1 - j} = \sin(\pi - j\theta) = h_j$, and hence $\Norm{h_{\text{odd}}} = \Norm{h_{\text{even}}}$.

 Let $z \in \R^{n-1}$ have entries $z_j = c^*/d_j' = c^* \varpi_{j+1} > 0$, so that $z_jd_j' = c^*$ for every $j$.  The matrix $T(z)$ as defined in \cite[Definition 3.1]{CLPS25} is the $r \times r$ matrix
 \[ T(z) = \begin{bmatrix} d_1 z_1 & & \\ d_2 z_2 & d_3 z_3 & \\ & d_4 z_4 & d_5 z_5 & \\ & & & \ddots \end{bmatrix};\]
 here, it is the $r \times r$ matrix with $c^*$ on the diagonal and subdiagonal and $0$ elsewhere.  Since $h_0 = h_{n+1} = 0$, $h_{j-1} + h_{j+1} = 2h_j\cos\theta$, and $2c^*\cos \theta = 1$, for $1 \le k \le r$, we have
 \[ \left( T(z) h_{\text{even}}\right)_k = c^*(h_{2k - 2} + h_{2k}) = h_{2k-1}, \quad \left( T(z)^t h_{\text{odd}}\right)_k = c^*(h_{2k -1} + h_{2k + 1}) = h_{2k}.\]
 Since $\Norm{h_{\text{odd}}} = \Norm{h_{\text{even}}}$, this says $T(z)y = x$ and $T(z)^tx = y$.  As $z$ has no zero entries, it is viable in the sense of \cite[Definition 3.8]{CLPS25}, so, by \cite[Proposition 3.19]{CLPS25}, $d^{D'}(u_1, u_n) = \Norm{\mu}\Norm{\nu}/d_1'$.  Finally, by \cite[Lemma 3.5]{CLPS25},
 \[ d_1' \sum_{j=1}^{n-1} z_j = \mu^t T(z) \nu = \Norm{\mu}\Norm{\nu} x^t T(z) y = \Norm{\mu}\Norm{\nu}x^t x = \Norm{\mu}\Norm{\nu},\]
 so \[d^{D'}(1,2) = d^{D'}(u_1, u_n) = \sum_{j = 1}^{n-1} z_j = c^* \sum_{i = 2}^n \varpi_i,\]
 and the result follows since $\varpi_1 = 0$.
\end{proof}

\begin{remark}
 \label{rem:simplification}
 In fact, $d^{D'}(1,2) = (n+1)/4$ in Lemma \ref{lem:sine}.  Indeed, since
 \[ \sin((i-1)\theta)\sin(i\theta) = \frac{1}{2}\left(\cos \theta - \cos((2i-1)\theta) \right)\]
 and
 \[ 2\sin\theta \sum_{i = 1}^n \cos((2i-1)\theta) = \sum_{i=1}^n \left(\sin(2i\theta) - \sin((2i -2)\theta) \right) = \sin(2n\theta) = -\sin(2\theta),\]
 we have
 \[ \sum_{i=1}^n \cos((2i - 1)\theta) = - \cos\theta, \]
 and hence
 \[ \sum_{i=1}^n \varpi_i = \frac{n\cos\theta}{2} - \frac{1}{2}\sum_{i=1}^n \cos((2i-1)\theta) = \frac{n\cos \theta}{2} + \frac{\cos \theta}{2} = \frac{(n+1)\cos\theta}{2},\]
 and, therefore,
 \[ d^{D'}(1,2) = c^* \sum_{i = 1}^n \varpi_i = \frac{n+1}{4}.\]
\end{remark}

\begin{proposition}
 \label{prop:witness}
 Let $\delta > 0$ and define $\gamma \colonequals c^*(1 + \delta)$.  If $\mathbf{1} \in \R^{n-1}$, $\chi(\delta) \colonequals -\Xi_r\left(L(\gamma \mathbf{1}) \right)$, $t(\delta) \colonequals \chi(\delta)/\gamma^{n-1}$, and the skew-symmetric matrix $M = (m_{ij}) \in M_n(\R)$ is defined by $m_{12} = t(\delta)$, $m_{23} = m_{34} = \cdots = m_{n,1} = -\gamma$, and $m_{i,j} = 0$ otherwise, then for sufficiently small $\delta > 0$, $t(\delta) > 0$ and $\Norm{M} < 1$.
\end{proposition}

\begin{proof}
 Let $M$ be as in the statement, $M' = L(-\gamma \mathbf{1})$ be the matrix obtained from $M$ by deleting the edge $\{1,2\}$, and let $M_S$ be the compression of $M$ to the set $S = \{3, \dots, n\}$; note that neither $M'$ nor $M_S$ depend on $t(\delta)$.  Since
 \[ m_{12}m{23} \cdots m_{n,1} = t(\delta) \cdot (-\gamma)^{n-1} = -t(\delta) \cdot \gamma^{n-1},\]
 the identity \eqref{eq:Xir} yields
 \begin{equation}
  \label{eq:Xir2}
  \Xi_r(M) = \Xi_r(M') - t(\delta)^2\Xi_{r-1}(M_S) + (-1)^r 2t(\delta)\gamma^{n-1}.
 \end{equation}
 By Lemma \ref{lem:constantpath}, the $k^\text{th}$ singular value of $M'$ is $s_k(M') = (1 + \delta)\cos(k\theta)/\cos(\theta)$, so, by the definition of $\Xi_r$, we have
 \[ \chi(\delta) = -\Xi_r(M') = -(1 - (1 + \delta)^2) \prod_{k=2}^r \left( 1 - (1 + \delta)^2\frac{\cos^2(k\theta)}{\cos^2 \theta}\right).\]
 The first factor is $-(1 - (1 + \delta)^2) = 2\delta + \delta^2$ and the others tend to $1 - \cos^2(k\theta)/\cos^2(\theta) > 0$ as $\delta \to 0$.  Hence, for small $\delta$, $\chi(\delta) > 0$ and so $t(\delta) > 0$; moreover, $\chi(\delta) \to 0$ and $t(\delta) \to 0$. Furthermore, \eqref{eq:Xir2} becomes
 \[ \Xi_r(M) = -\chi(\delta) + 2\chi(\delta) - t(\delta)^2 \Xi_{r-1}(M_S) = \chi(\delta) \left(1 - \frac{\chi(\delta)\Xi_{r-1}(M_S)}{\gamma^{2(n-1)}} \right),\]
 and $(1 - \chi(\delta)\Xi_{r-1}(M_S)/\gamma^{2(n-1)} \to 1$ as $\delta \to 0$, so $\Xi_r(M) > 0$ for small $\delta$.

 Finally, if $X$ is skew-symmetric, then $\ii X$ is Hermitian, and applying Weyl's inequality \cite[Theorem 4.3.1]{HJ13} to $\ii X$, whose second largest eigenvalue is the second largest singular value $s_2(X)$ of $X$, we see that the map $s_2$ sending $X$ to its second largest singular value is $1$-Lipschitz for the operator norm.  Since $M - L(-c*\mathbf{1}) = t(\delta)\left(E_{1,2} - E_{2,1}\right) + L(-c^*\delta\mathbf{1})$ has norm at most $t(\delta) + 2c^*\delta \cos \theta = t(\delta) + \delta$, Lemma \ref{lem:constantpath} gives
 \[ s_2(M) \le \frac{\cos(2\theta)}{\cos \theta} + t(\delta) + \delta < 1\]
 for small $\delta$.  Then, for $j \ge 2$ the factors $1 - s_j(M)^2$ of $\Xi_r(M) = \prod_{j = 1}^r \left(1 - s_j(M)^2\right)$ are positive, and so $\Xi_r(M) > 0$ forces $1 - s_1^2(M) > 0$, i.e., $\Norm{M} < 1$.
\end{proof}

\begin{theorem}
 \label{thm:0mod4}
 Let $n \ge 4$ be divisible by $4$.  There exists a choice of weights for the $n$-cycle such that, if $D$ denotes the Dirac operator of the $n$-cycle, $e$ is the edge $\{1,2\}$, and $D'$ is the Dirac operator obtained by deleting $e$, then
 \[ d^{D'}(1,2) < \dD(1,2).\]
\end{theorem}

\begin{proof}
 Let $\delta > 0$ and define $\gamma$ and $t(\delta)$ as in the statement of Proposition \ref{prop:witness}.  Define weights $w_{j,j+1}$ for the edge $\{j, j+1\}$ of the $n$-cycle by $w_{i,i+1} = \varpi_i = \sin((i-1)\theta)\sin(i\theta)$ for $2 \le i \le n$ and
 \[w_{12} = \frac{c^*(1+\delta)}{t(\delta)}\sum_{i = 2}^n \varpi_i = \frac{(1 + \delta)(n+1)}{4t(\delta)},\]
 where the last equality follows from Remark \ref{rem:simplification}.  Define $M = (m_{i,j})$ as in the statement of Proposition \ref{prop:witness}.  Then, for sufficiently small $\delta > 0$, $\Norm{M} < 1$, and, by the choice of $w_{12}$,
 \[ \sum_{i = 1}^n w_{i,i+1} m_{i,i+1} = w_{12} t(\delta) - \gamma \sum_{i = 2}^n \varpi_i = \frac{(1 + \delta)(n+1)}{4} - \frac{(1 + \delta)(n+1)}{4} = 0.\]  Thus, by Lemma \ref{lem:weighting}, there exists an admissible diagonal matrix $A$ such that $[A,D] = M$, and, moreover,
 \[ \dD(1,2) \ge w_{12} m_{12} = \frac{(1 + \delta)(n+1)}{4} > \frac{n+1}{4} = d^{D'}(1,2),\]
 where the last equality follows from Remark \ref{rem:simplification}.  The result follows.
\end{proof}


\section{Deletion of vertices or edges that are far away can affect noncommutative distance}
\label{sect:far}

As we have seen, deleting an edge can affect the noncommutative distance between vertices $x$ and $y$.  In all examples up to this point, the deleted edge was $\{x,y\}$.  We will show in this section that the deleted edge can actually be ``far'' from both $x$ and $y$ in both the sense of weights and of graph-theoretic distance.

\subsection{When edge deletion has no effect.}

The following is a formal consequence of \cite[Theorem 4, (9)]{Bes21}; we note it because it constrains the structure of counterexamples to \eqref{eq:question}.

\begin{proposition}
 \label{prop:safe}
 Let $e \in E$, $x, y \in V$, $D$ be a Dirac operator for $G$, and $D'$ the Dirac operator for $G - e$ obtained after deleting $e$.  Suppose that, after pruning, the path space $\PP_{G - e}(x,y)$ is an induced subgraph of $G$.  Then, $d^{D'}(x,y) \ge \dD(x,y)$.
\end{proposition}

\begin{proof}
 By Proposition \ref{prop:prune}, $d^{D'}(x,y) = d^{D''}(x,y)$, where $D''$ is induced on $\PP_{G-e}(x,y)$.  By Proposition \ref{prop:deletion}, the distance on an induced subgraph of $G$ is at least $\dD(x,y)$.
\end{proof}

\begin{corollary}
 \label{cor:necessary}
 Let $e \in E(G)$, $x, y \in V(G)$, $D$ be a Dirac operator for $G$, and $D'$ the Dirac operator for $G - e$ obtained after deleting $e$.  If $d^{D'}(x,y) < \dD(x,y)$, then both endpoints of $e$ lie in $\PP_{G-e}(x,y)$.
\end{corollary}

\begin{proof}
 Let $S$ be the vertex set of $\PP_{G-e}(x,y)$, so that $\PP_{G-e}(x,y) = (G - e)[S]$.  Since $D$ and $D'$ differ only in the two entries corresponding to $e$, if $e$ does not have both endpoints in $S$, then the compressions $D_S'$ and $D_S$ coincide. Equivalently, $(G - e)[S] = G[S]$, so $\PP_{G-e}(x,y)$ is an induced subgraph of $G$.  The result follows from Proposition \ref{prop:safe}.
\end{proof}

\begin{example}
 If $G$ is the $4$-cycle $1 - 2 - 3 - 4 - 1$, the deletion of $e = \{1,2\}$ leaves $\PP_{G-e}(1,3)$ equal to the path $1 - 4 - 3$, which does not contain the vertex $2$; similarly, $\PP_{G-e}(1,4)$ does not contain $2$, $\PP_{G-e}(2,3)$ and $\PP_{G-e}(2,4)$ do not contain $1$, and $\PP_{G-e}(3,4)$ contains neither $1$ nor $2$.  Hence, by Corollary \ref{cor:necessary}, the distance can only possibly decrease between the vertices $1$ and $2$, i.e., in a $4$-cycle, the distance between two vertices can only possibly decrease between the two endpoints of the deleted edge.
\end{example}

\subsection{Deletion of edges far away in the graph-theoretic metric}

We will now study the case when the deleted edge is far away with respect to graph-theoretic distance, i.e., the number of edges.

\begin{proposition}
 \label{prop:pendant}
 Let $G$ be a weighted graph, $e \in E(G)$, $x, y \in V(G)$, $D$ the Dirac operator of $G$, $D'$ the Dirac operator of $G - e$, and suppose
 \[ \dD(x,y) - d^{D'}(x,y) > 0.\]
 Let $G^+$ be obtained from $G$ by attaching a path $P_x$ at $x$, ending at a new vertex $x'$, and a path $P_y$ at $y$, ending at a new vertex $y'$, where all vertices of $P_x$ and $P_y$ other than $x$ and $y$ are new and distinct; let $\delta_x$ and $\delta_y$ denote the total weights of $P_x$ and $P_y$, let $D^+$ be the Dirac operator of $G^+$, and let $D''$ be the Dirac operator of $G^+ - e$.  If $\delta_x + \delta_y < \dD(x,y) - d^{D'}(x,y)$, then
 \[ d^{D+}(x', y') \ge \dD(x,y) > d^{D'}(x,y) + \delta_x + \delta_y \ge d^{D''}(x',y').\]
 In particular, \eqref{eq:question} also fails for $G^+$, $e$, and the pair $(x', y')$.
\end{proposition}

\begin{proof}
 The second inequality follows immediately from the hypotheses.  We consider next the first inequality.  Let $A = \diag(a)$ be admissible for $G$ with $a_x - a_y = \dD(x,y)$, and extend $a$ to $a^+$ by assigning the value $a_x$ to every vertex of $P_x$ and the value $a_y$ to every vertex of $P_y$, and set $A^+ = \diag(a^+)$.  Every new commutator entry is of the form $(a_u - a_v)d_{uv}$ with $a_u = a_v$ and hence vanishes, so $[A^+, D^+]$ is $[A,D]$ bordered by zero rows and columns, and furthermore $\Norm{[A^+, D^+]} = \Norm{[A,D]} \le 1$.  Since $a_{x'} - a_{y'} = a_x - a_y$, the claim follows from \eqref{eq:def}.

 For the third inequality, we note that no simple path of $G^+ - e$ from $x$ to $y$ enters $P_x$ or $P_y$, so $\PP_{G^+ - e}(x,y) = \PP_{G - e}(x,y)$ and, by Proposition \ref{prop:prune}, $d^{D''}(x,y) = d^{D'}(x,y)$.  Furthermore, by Proposition \ref{prop:geod}, $d^{D''}(x,x') \le \ell^{D''}(x,x') \le \delta_x$ and $d^{D''}(y,y') \le \ell^{D''}(y,y') \le \delta_y$.  The triangle inequality for $d^{D''}$ now yields $d^{D''}(x',y') \le \delta_x + d^{D'}(x,y) + \delta_y$, completing the proof.
\end{proof}

\begin{corollary}
 \label{cor:nonincident}
 There exists a weighted graph $G^+$ with Dirac operator $D^+$, edge $e$, and vertices $x,y$ not incident with $e$ such that, if $D''$ is the Dirac operator of $G^+ - e$, we have
 \[ d^{D^+}(x,y) > d^{D''}(x,y).\]
\end{corollary}

\begin{proof}
 Let $G^+$ have vertices $x, 1, 2, 3, 4, y$ and weights $w_{x,1} = w_{2,y} = 1/10$, $w_{12} = 5$, and $w_{23} = w_{34} = w_{41} = 1$, and let $e = \{1,2\}$.  Consider first the induced subgraph $G$ on $\{1,2,3,4\}$; it is a $4$-cycle, and we will denote the Dirac operator of $G$ by $D$.  By Theorem \ref{thm:main} with $\ep = 1/5$, we have $\dD(1,2) \ge \sqrt{5}$ and $d^{D'}(1,2) = 2$, where $D'$ is the Dirac operator of $G - e$.  Noting that $\sqrt{5} - 2 > 1/5$ and that $P_x$, $P_y$ each consist of a single edge, we can apply Proposition \ref{prop:pendant} by giving the weight $\delta_x = \delta_y = 1/10$ to each of $\{x,1\}$ and $\{2,y\}$.  The result follows.

 \end{proof}

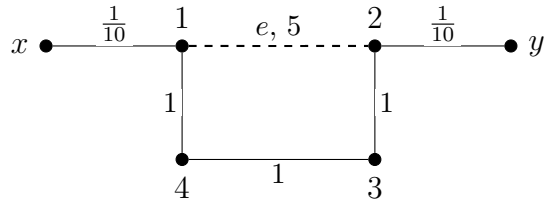
\begin{figure}[h]
\centering
\begin{tikzpicture}[scale=1.5]
  \node[vtx,label=left:{$x$}]   (x)  at (-1.2,1) {};
  \node[vtx,label=above:{$1$}]  (v1) at (0,1) {};
  \node[vtx,label=above:{$2$}]  (v2) at (1.7,1) {};
  \node[vtx,label=below:{$3$}]  (v3) at (1.7,0) {};
  \node[vtx,label=below:{$4$}]  (v4) at (0,0) {};
  \node[vtx,label=right:{$y$}]  (y)  at (2.9,1) {};
  \draw (x)--(v1) node[wt,midway,above]{$\tfrac1{10}$};
  \draw[del] (v1)--(v2) node[wt,midway,above]{$e$, $5$};
  \draw (v2)--(y) node[wt,midway,above]{$\tfrac1{10}$};
  \draw (v2)--(v3) node[wt,midway,right]{$1$};
  \draw (v3)--(v4) node[wt,midway,below]{$1$};
  \draw (v4)--(v1) node[wt,midway,left]{$1$};
\end{tikzpicture}
\medskip
\caption{Corollary~\ref{cor:nonincident}: deleting the dashed edge $e$, which is incident
to neither $x$ nor $y$, strictly decreases $d(x,y)$ (from $\ge\sqrt5$ to $\le\frac{11}5$).
Replacing the two pendant edges by paths of $N$ edges of weight $\frac1{10N}$ pushes $e$
arbitrarily far from $x$ and $y$ in the graph-theoretic metric
(Corollary~\ref{cor:far}).}
\label{fig:nonincident}
\end{figure}

\begin{corollary}
 \label{cor:far}
 For every positive integer $N$, there exists a weighted graph $G^+$ with Dirac operator $D^+$, edge $e$, and vertices $x,y$ at graph-theoretic distance at least $N$ from $e$ such that, if $D''$ is the Dirac operator of $G^+ - e$, we have
 \[ d^{D^+}(x,y) > d^{D''}(x,y).\]
\end{corollary}

\begin{proof}
 In the proof of Corollary \ref{cor:nonincident}, we replace each pendant edge by a path of $N$ edges of weight $\frac{1}{10N}$ each.  The total weights are unchanged, and so Proposition \ref{prop:pendant} still applies.
\end{proof}

\begin{remark}
 The hypothesis $\delta_x + \delta_y < \dD(x,y) - d^{D'}(x,y)$ of Proposition \ref{prop:pendant} means that the deleted edge $e$ in Corollaries \ref{cor:nonincident} and \ref{cor:far} remains close to $x$ and $y$ in the weighted metric.  We are able to remove this restriction in Subsection \ref{ss:farweight} at the cost of adding more vertices and letting the maximum degree of the graph grow without bound.
\end{remark}

\subsection{Failure of locality.}

In \cite[Section 5]{Bes21}, it is posited that a vertex $z$ that is geodesically very far from $x$ and $y$ should not matter much, if at all, in the computation of $\dD(x,y)$.  This subsection is devoted to showing that this is not actually the case.  We begin with a lemma that is key to these computations.

\begin{lemma}[blow up]
 \label{lem:blowup}
 Let $H$ be a finite simple graph with vertices $h_0, \dots, h_r$ and Dirac entries $C_{ij} \ge 0$, and let $k_0, \dots, k_r$ be positive integers.  Let $G = H[k_0, \dots, k_r]$ be the graph with vertex set $\bigsqcup_{i=0}^r L_i$, where $|L_i| = k_i$ for each $i$ in which $L_i$ and $L_j$ are completely joined whenever $h_i \sim h_j$, every such edge has Dirac entry $C_{ij}$, and with no edges inside $L_i$ or between $L_i$ and $L_j$ when $h_i \not\sim h_j$; let $G$ have Dirac operator $D$.  Let $\tilde{H}$ be the graph $H$ with Dirac operator $\tilde{C}$ having entries $\tilde{C}_{ij} = \sqrt{k_i k_j} C_{ij}$.  If two of the $L_i$ are singletons, say $L_p = \{x\}$ and $L_{p'} = \{y\}$, then
 \[ \dD(x,y) = d^{\tilde{C}}(h_p, h_{p'}).\]
\end{lemma}

\begin{proof}
 Let $\Gamma = \prod_{i = 0}^r \Sym(L_i) \le \Aut(G)$; note that $\Gamma$ fixes $x$ and $y$, since $L_p$ and $L_{p'}$ are singletons.  The set
 \[ \calK = \{ a \in \R^{V(G)}: \Norm{[\diag(a),D]} \le 1\}\] is convex and invariant under the action of $\Gamma$, and the objective $a_x - a_y$ is $\Gamma$-invariant.  Since $H$ is connected, so is $G$, and the supremum in \eqref{eq:def} is attained (see, for instance, \cite[Section 2.3]{CLPS25}); hence, averaging a maximizing $a$ over $\Gamma$ produces a maximizing $\overline{a} \in \calK$ that is constant on each $L_i$.

 Let $a$ be constant on each $L_i$ with value $\alpha_i$ and let $M = (m_{uv}) = [A,D]$, so $m_{uv} = (\alpha_i - \alpha_j)C_{ij}$ for $u \in L_i$ and $v \in L_j$.  If $e_u$ is the standard basis vector of $\R^{V(G)}$ at the vertex $u$, we define
 \[ b_i \colonequals k_i^{-1/2} \sum_{u \in L_i} e_u,\]
 so $b_i$ is the normalized indicator vector of $L_i$, and define $U \colonequals {\rm span}\{b_0, \dots, b_r\}$.  Thus, $U$ consists of the vectors that are constant on each $L_i$, and $U^\perp$ consists of the vectors whose entries sum to zero over each $L_i$; that is, $U^\perp$ is spanned by the vectors $\xi \in \R^{V(G)}$ that are supported on a single layer $L_i$ and satisfy $\sum_{u \in L_i} \xi_u = 0$.

 Let $\xi$ be such a vector, supported on $L_i$.  For $v \in L_j$,
 \[ (M\xi)_v = \sum_{u \in L_i} m_{vu}\xi_u = (\alpha_j - \alpha_i)C_{ij} \sum_{u \in L_i} \xi_u = 0,\] and, for $v \in L_i$, $(M\xi)_v = 0$ as well, since there are no edges inside $L_i$.  Hence, $M|_{U^\perp} = 0$.  Moreover, for $v \in L_j$,
 \[ (Mb_i)_v = k_i^{-1/2} \sum_{u \in L_i} m_{vu} = k_i^{-1/2} \cdot k_i(\alpha_j -\alpha_i)C_{ij} = k_i^{1/2}(\alpha_j - \alpha_i)C_{ij},\]
 which is constant on $L_j$, so
 \[ Mb_i = \sum_{j = 0}^r \sqrt{k_i k_j} (\alpha_j - \alpha_i) C_{ij}b_j.\]
 Thus, $U$ is $M$-invariant, and, since $\R^{V(G)} = U \oplus U^\perp$ and $M|_{U^\perp} = 0$, $\Norm{M} = \Norm{M|_U}$.  In the orthonormal basis $b_0, \dots, b_r$, the $j,i$-entry of $M|_U$ is $\sqrt{k_i k_j} (\alpha_j - \alpha_i) C_{ij}$, so $M|_U$ is exactly the commutator of $\diag(\alpha_0, \dots, \alpha_r)$ with the Dirac operator $\tilde{C}$ of $\tilde{H}$, whose entries are $\sqrt{k_ik_j} C_{ij}$.  Finally, $a_x - a_y = \alpha_p - \alpha_{p'}$ because $L_p = \{x\}$ and $L_{p'} = \{y\}$, and every $\alpha \in \R^{r+1}$ arises from an $a$ that is constant on each $L_i$.  Therefore, $\dD(x,y) = d^{\tilde{C}}(h_p, h_{p'}).$
\end{proof}

\begin{theorem}
 \label{thm:nonlocal}
 Let $q \ge 1$ and $k \ge 1$.  Let $G = G_{q,k}$ be the graph obtained from the path
 \[ x - A_1 - A_2 - \cdots - A_q - z - B_q - \cdots - B_1 - y\]
 by blowing up each $A_j$ and $B_j$ to $k$ vertices as in Lemma \ref{lem:blowup}, where each edge of the blow up graph has weight $1$, and by further adjoining the edge $\{x,y\}$ of weight $1$.  Then, if $D$ is the Dirac operator of $G$ and $V' = V(G)\backslash\{z\}$,
 \begin{itemize}
  \item[(i)] the graph-theoretic distance and the weighted distance from $z$ to each of $x,y$ equal $q+1$;
  \item[(ii)] $\dD(x,y) \le \frac{4}{\sqrt{k}} + \frac{2(q-1)}{k}$;
  \item[(iii)] $d_{V'}^D(x,y) = 1$.
 \end{itemize}
In particular, for any $\rho > 0$ and $\eta > 0$, we may choose $q$ with $q+1 \ge \rho$ and then $k$ large enough that $\dD(x,y) < \eta$ while $d_{V'}^D(x,y) = 1$; that is, deletion of a single vertex at arbitrarily large graph-theoretic and weighted distance from $x$ and $y$ changes $\dD(x,y)$ by an arbitrarily large factor.
\end{theorem}

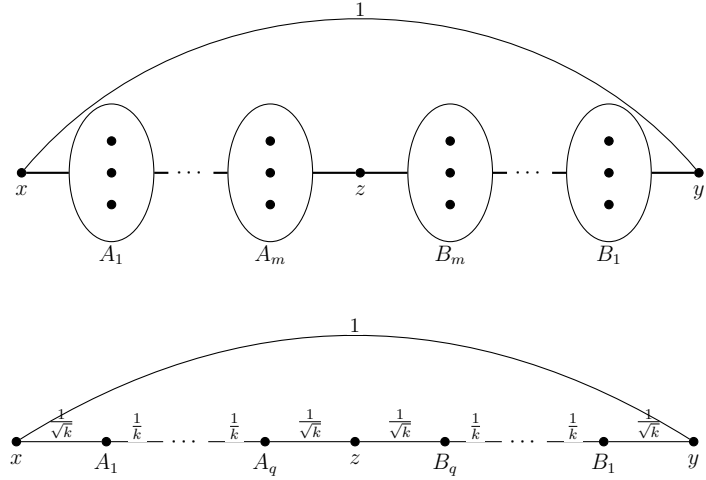
\begin{figure}[h]
\centering
\begin{tikzpicture}[scale=0.70, every node/.style={transform shape}]
  \node[vtx,label=below:{$x$}] (x) at (0,0) {};
  \node[lyr] (A1) at (1.7,0) {};
  \node at (1.7,-1.55) {$A_1$};
  \node (d1) at (3.2,0) {$\cdots$};
  \node[lyr] (Am) at (4.7,0) {};
  \node at (4.7,-1.55) {$A_m$};
  \node[vtx,label=below:{$z$}] (z) at (6.4,0) {};
  \node[lyr] (Bm) at (8.1,0) {};
  \node at (8.1,-1.55) {$B_m$};
  \node (d2) at (9.6,0) {$\cdots$};
  \node[lyr] (B1) at (11.1,0) {};
  \node at (11.1,-1.55) {$B_1$};
  \node[vtx,label=below:{$y$}] (y) at (12.8,0) {};
  \foreach \x in {1.7,4.7,8.1,11.1}{\foreach \y in {0.6,0,-0.6}{\node[vtx] at (\x,\y) {};}}
  \draw[thick] (x)--(A1); \draw[thick] (A1)--(d1); \draw[thick] (d1)--(Am);
  \draw[thick] (Am)--(z); \draw[thick] (z)--(Bm); \draw[thick] (Bm)--(d2);
  \draw[thick] (d2)--(B1); \draw[thick] (B1)--(y);
  \draw (x) to[bend left=50] node[wt,midway,above]{$1$} (y);
\end{tikzpicture}

\bigskip

\begin{tikzpicture}[scale=0.70, every node/.style={transform shape}]
  \node[vtx,label=below:{$x$}] (x) at (0,0) {};
  \node[vtx,label=below:{$A_1$}] (A1) at (1.7,0) {};
  \node (d1) at (3.2,0) {$\cdots$};
  \node[vtx,label=below:{$A_q$}] (Am) at (4.7,0) {};
  \node[vtx,label=below:{$z$}] (z) at (6.4,0) {};
  \node[vtx,label=below:{$B_q$}] (Bm) at (8.1,0) {};
  \node (d2) at (9.6,0) {$\cdots$};
  \node[vtx,label=below:{$B_1$}] (B1) at (11.1,0) {};
  \node[vtx,label=below:{$y$}] (y) at (12.8,0) {};
  \draw (x)--(A1) node[wt,midway,above]{$\tfrac1{\sqrt k}$};
  \draw (A1)--(d1) node[wt,midway,above]{$\tfrac1k$};
  \draw (d1)--(Am) node[wt,midway,above]{$\tfrac1k$};
  \draw (Am)--(z) node[wt,midway,above]{$\tfrac1{\sqrt k}$};
  \draw (z)--(Bm) node[wt,midway,above]{$\tfrac1{\sqrt k}$};
  \draw (Bm)--(d2) node[wt,midway,above]{$\tfrac1k$};
  \draw (d2)--(B1) node[wt,midway,above]{$\tfrac1k$};
  \draw (B1)--(y) node[wt,midway,above]{$\tfrac1{\sqrt k}$};
  \draw (x) to[bend left=32] node[wt,midway,above]{$1$} (y);
\end{tikzpicture}
\medskip
\caption{Above: the graph $G_{q,k}$ of Theorem~\ref{thm:nonlocal}. Each ellipse is an
independent set of $k$ vertices, each thick segment denotes a complete join in which
every edge has weight $1$, and the arc is the chord $\{x,y\}$ of weight $1$. Below: the
quotient $\widetilde H$ of Lemma~\ref{lem:blowup}, in which $d(x,y)$ is unchanged;
singleton-to-layer joins acquire a factor $\sqrt k$ and layer-to-layer joins a factor
$k$, so the weights shrink as shown. Deleting the single vertex $z$ makes the two halves
pendant trees and restores $d(x,y)=1$.}
\label{fig:blowup}
\end{figure}

\begin{proof}
 First, (i) follows from the fact that all edges in $G$ have weight $1$, and every path from either $x$ or $y$ to $z$ uses exactly $q+1$ of the edges.

 For (ii), by Lemma \ref{lem:blowup} with layer sizes of $k$ for all $A_i, B_i$ and $1$ for $x, y, z$, we have $\dD(x,y) = d^{\tilde{C}}(x,y)$, where $\tilde{C}$ is the Dirac operator of the cycle
 \[ x \overset{1/\sqrt{k}}{-} A_1 \overset{1/k}{-} \cdots \overset{1/k}{-} A_q \overset{1/\sqrt{k}}{-} z \overset{1/\sqrt{k}}{-} B_q \overset{1/k}{-} \cdots \overset{1/k}{-} B_1 \overset{1/\sqrt{k}}{-} y \overset{1}{-} x,\]
 since edges adjoining a singleton to a set of size $k$ acquire the factor $\sqrt{k}$, and edges adjoining two sets of size $k$ acquire the factor $k$.  By Proposition \ref{prop:geod}, \[d^{\tilde{C}}(x,y) \le \ell^{\tilde{C}}(x,y) \le \frac{4}{\sqrt{k}} + \frac{2(q-1)}{k}.\]

 Finally, (iii) follows from Proposition \ref{prop:prune}, noting that $\PP_{G-z}(x,y) = \{x,y\}$ and $w_{xy} = 1$.
\end{proof}

\begin{remark}
 A weak form of locality does in fact hold once the maximal degree of the graph is controlled: a distant vertex can change a noncommutative distance only by a bounded factor.  Indeed, if $z$ is far enough from $x$ and $y$ that $\ell_{V'}^{D}(x,y) = \ell^D(x,y)$, then combining Proposition \ref{prop:geod} with \cite[Theorem 5]{Bes21} gives
 \[ \frac{d_{V'}^D(x,y)}{\dD(x,y)} \le \frac{\ell_{V'}{D}(x,y)}{\ell^D(x,y)/\Delta(G)} = \Delta(G),\]
 where $\Delta(G)$ denotes the maximum degree of $G$.

 In Theorem \ref{thm:nonlocal}, we have $\Delta(G_{q,k}) = 2k$, and the ratio is on the order of $\sqrt{\Delta(G)}$ for $q$ fixed.  Indeed, assume $k \ge (2q + 2)^2$ so that $4/\sqrt{k} + 2(q-1)/k \le 1$.  By Lemma \ref{lem:blowup}, $\dD(x,y) = d^{\tilde{C}}(x,y)$, where $\tilde{C}$ is the weighted cycle in the proof of Theorem \ref{thm:nonlocal}.  Since $\Delta(\tilde{C}) = 2$ and \[\ell^{\tilde{C}}(x,y) = \min(1, 4/\sqrt{k} + 2(q-1)/k) = 4/\sqrt{k} + 2(q-1)/k,\]
 by \cite[Theorem 5]{Bes21} and Proposition \ref{prop:geod}, we have
 \[ 2/\sqrt{k} + (q-1)/k \le \dD(x,y) \le 4/\sqrt{k} + 2(q-1)/k.\]  Finally, since
 \[ \frac{4}{\sqrt{k}} \le 4/\sqrt{k} + 2(q-1)/k \le \frac{2q + 2}{\sqrt{k}},\]
 it follows that
 \[ \frac{\sqrt{\Delta(G)}}{(2q + 2)\sqrt{2}} \le \frac{d_{V'}^D(x,y)}{\dD(x,y)} \le \frac{\sqrt{\Delta(G)}}{2\sqrt{2}}, \]
i.e., the ratio is on the order of $\sqrt{\Delta(G)}$ (as opposed to $\Delta(G)$).  It would be interesting to know whether the bound of $\Delta(G)$ can actually be attained or whether the correct bound happens to be $\sqrt{\Delta(G)}$, and it would be interesting to know whether the ratio tends to one as the distance from $z$ to $x$ and $y$ goes to infinity when $\Delta(G)$ is bounded.
\end{remark}

\subsection{Deletion of edges far away in the weighted metric.}
\label{ss:farweight}

Proposition \ref{prop:pendant} allows for the deleted edge to be arbitrarily far from the vertices $x$ and $y$ in the graph-theoretic metric (see Corollary \ref{cor:far}), but the hypothesis $\delta_x + \delta_y < \dD(x,y) - d^{D'}(x,y)$ bounds the total weights of the attaching paths; hence, the edge $e$ stays close to $x$ and $y$ in the weighted metric.  However, combining Proposition \ref{prop:pendant} with Lemma \ref{lem:blowup} allows us to remove the restriction in the weighted metric.

\begin{theorem}
 \label{thm:farweight}
 Let $W > 0$ and let $k \ge 400W^2$ be an integer.  Let $G = G(W,k)$ have vertices $x,y$, four further vertices $1,2,3,4$, and two independent sets $L_1$, $L_2$ of size $k$, with edges defined as follows:
 \begin{itemize}
  \item $x$ is joined to every vertex of $L_1$, and every vertex of $L_1$ is joined to $1$, all of weight $W$;
  \item $y$ is joined to every vertex of $L_2$, and every vertex of $L_2$ is joined to $2$, all of weight $W$;
  \item the vertices $\{1,2,3,4\}$ form a $4$-cycle $1 - 2 - 3 - 4 - 1$ with $w_{12} = 5$ and $w_{23} = w_{34} = w_{41} = 1$.
 \end{itemize}
 Let $e = \{1,2\}$.  Then, $e$ is incident to neither $x$ nor $y$, the vertex $1$ is at weighted distance $2W$ from $x$, the vertex $2$ is at weighted distance $2W$ from $y$, and, if $G$ has Dirac operator $D$ and $G - e$ has Dirac operator $D'$,
 \[ \dD(x,y) \ge \sqrt{5} > \frac{11}{5} \ge d^{D'}(x,y).\]
 In particular, taking $W = \rho/2$ and $k = \lceil 100 \rho^2 \rceil$, the deleted edge $e$ lies at weighted distance $\rho$ from both $x$ and $y$, and $\rho$ may be arbitrarily large.
\end{theorem}

\begin{proof}
 The vertex set is partitioned into $\{x\}$, $L_1$, $\{1\}$, $\{2\}$, $\{3\}$, $\{4\}$, $L_2$, $\{y\}$, no set contains an edge, and sets adjacent in $H$ are completely joined with edges of constant weight; thus, Lemma \ref{lem:blowup} applies with $k_{L_1} = k_{L_2} = k$ and all other $k_i = 1$.  Since $x$ and $y$ lie in singleton sets, $\dD(x,y) = d^{C}(x,y)$, where $C$ is the Dirac operator of $H$, which is the quotient
 \[
  x\overset{W/\sqrt k}{-\!\!-}L_1\overset{W/\sqrt k}{-\!\!-}1
  \ \bigl[\,\text{$4$-cycle }5,1,1,1\,\bigr]\
  2\overset{W/\sqrt k}{-\!\!-}L_2\overset{W/\sqrt k}{-\!\!-}y .
\]
The edge $e$ joins two singleton sets, so $G - e$ is the blow-up of $H - e$ with the same set sizes, and, likewise, $d^{D'}(x,y) = d^{C'}(x,y)$, where $D'$ and $C'$ denote the Dirac operators of $G - e$ and $H - e$, respectively.

It hence suffices to prove the two inequalities in $H$.  Let $G_0$ be the $4$-cycle on $\{1,2,3,4\}$ with the stated weights; let $D_0$ be its Dirac operator, and let $D_0'$ be the Dirac operator of $G_0 - e$.  By Theorem \ref{thm:main} with $\ep = 1/5$ we have $d^{D_0}(1,2) \ge \sqrt{5}$ and $d^{D_0'}(1,2) = 2$, so $d^{D_0}(1,2) - d^{D_0'}(1,2) > 1/5$.  The graph $H$ is obtained from $G_0$ by attaching at $1$ the path $1 - L_1 - x$ and at $2$ the path $2 - L_2 - y$, of total weights $\delta_x = \delta_y = 2W/\sqrt{k} \le 1/10$ (since $k \ge 400W^2$); hence, \[\delta_x + \delta_y \le \frac{1}{5} < d^{D_0}(1,2) - d^{D_0'}(1,2),\]
and, by Proposition \ref{prop:pendant}, we have
\[ d^{C'}(x,y) \le d^{D_0'}(1,2) + \delta_x + \delta_y \le \frac{11}{5} < \sqrt{5} \le d^{D_0}(1,2) \le d^{C}(x,y).\]
Finally, the only neighbors of $x$ are the vertices of $L_1$, each of whose only other neighbor is $1$, so $\ell^D(x,1) = 2W$, and, similarly, $\ell^D(y,2) = 2W$.
\end{proof}

\begin{figure}[h]
\centering
\begin{tikzpicture}[scale=0.85, every node/.style={transform shape}]
  \node[vtx,label=below:{$x$}] (x) at (0,0) {};
  \node[lyr] (L) at (1.6,0) {};
  \node at (1.6,-1.5) {$L_1$};
  \foreach \y in {0.6,0,-0.6}{\node[vtx] at (1.6,\y) {};}
  \node[vtx,label=below left:{$1$}] (v1) at (3.2,0) {};
  \node[vtx,label=below right:{$2$}] (v2) at (5.6,0) {};
  \node[vtx,label=above:{$4$}] (v4) at (3.9,1.3) {};
  \node[vtx,label=above:{$3$}] (v3) at (4.9,1.3) {};
  \node[lyr] (M) at (7.2,0) {};
  \node at (7.2,-1.5) {$L_2$};
  \foreach \y in {0.6,0,-0.6}{\node[vtx] at (7.2,\y) {};}
  \node[vtx,label=below:{$y$}] (y) at (8.8,0) {};
  \draw[thick] (x)--(L) node[wt,midway,above]{$W$};
  \draw[thick] (L)--(v1) node[wt,midway,above]{$W$};
  \draw[thick] (v2)--(M) node[wt,midway,above]{$W$};
  \draw[thick] (M)--(y) node[wt,midway,above]{$W$};
  \draw[del] (v1)--(v2) node[wt,midway,below]{$e$, $5$};
  \draw (v1)--(v4) node[wt,midway,left]{$1$};
  \draw (v4)--(v3) node[wt,midway,above]{$1$};
  \draw (v3)--(v2) node[wt,midway,right]{$1$};
\end{tikzpicture}
\medskip
\caption{Theorem~\ref{thm:farweight}. Each ellipse is an independent set of $k$ vertices
and each thick segment a complete join, every edge of which has weight $W$. Deleting the
dashed edge $e$, at weighted distance $2W$ from both $x$ and $y$, strictly decreases
$d(x,y)$.}
\label{fig:farweight}
\end{figure}
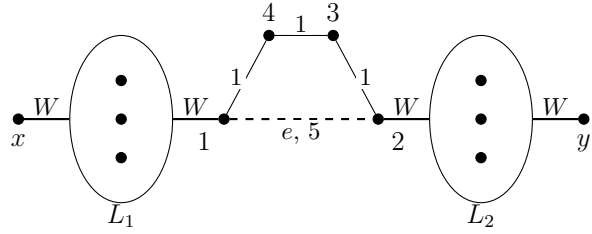

\begin{remark}
 \label{rem:combined}
 Both the graph-theoretic and the weighted distance can be made large simultaneously: replacing each of $L_1$ and $L_2$ by a chain of $q$ sets, each of size $k$, as in Theorem \ref{thm:nonlocal} gives attaching paths of total weight $2W/\sqrt{k} + (q-1)W/k$ in the quotient while $e$ sits at graph-theoretic distance $q+1$ and weighted distance $(q+1)W$ from each of $x$ and $y$.  Again, this construction requires both $x$ and $y$ to have large degree.
\end{remark}

\section*{Declaration of Generative AI Use}

Claude Opus 5 and Fable 5 were used to draw the figures and derive the examples.  Initial versions of the proofs in this manuscript were conceived by Claude Opus 5, Fable 5, and Opus 5.5 in a series of chats and subsequently verified independently by the author.  The present draft was written by the author, and the author assumes responsibility for all content in the paper.

\end{document}